\documentclass[10pt]{article}
\usepackage{amssymb,amsfonts}
\usepackage{pb-diagram}
\usepackage{amsmath,amsthm,amsfonts,amssymb}
\usepackage[usenames]{color}
\usepackage{hyperref}
\usepackage{soul}
\usepackage{graphicx}
\usepackage{amssymb}
\usepackage{amsmath}
\usepackage{amssymb}

\newtheorem{theorem}{Theorem}
\newtheorem{definition}{Definition}

\newtheorem{lemma}{Lemma}
\newtheorem{cor}{Corollary}
\newtheorem{remark}{Remark}

\newtheorem{prop}{Proposition}

\newcommand{\be}{\begin{enumerate}}
\newcommand{\ee}{\end{enumerate}}
\newcommand{\beq}{\begin{equation}}
\newcommand{\eeq}{\end{equation}}

\def\N{{\mathbb{N}}}
\def\Z{{\mathbb{Z}}}

\def\MN{{\mathbb{N}}}

\def\MA{{\mathbb{A}}}
\def\MB{{\mathbb{B}}}
\def\MM{{\mathbb{M}}}

{}
{}

\title{What does a group algebra of a free group ``know'' about the group?}
\author{Olga Kharlampovich\footnote{Hunter College, CUNY, Supported by the PSC-CUNY  award, jointly funded by The Professional Staff Congress and The City University of New York and by a grant from the Simons Foundation. }  and Alexei Myasnikov \footnote{Stevens Institute of Technology, supported by NSF grant DMS-1201379}}

\date{}

\begin{document}

\maketitle

\begin{abstract} 
We describe  solutions to the  problem of elementary classification in the class of group algebras of free groups. We will show that unlike free groups, two group algebras of free groups over infinite fields are elementarily equivalent if and only if the groups are isomorphic and the fields are equivalent in the weak second order logic.   We will show that  the set of all free bases of a free group $F$ is 0-definable in the group algebra $K(F)$ when $K$ is an infinite field,  the set of geodesics is definable, and many geometric  properties of $F$ are definable in $K(F)$. Therefore $K(F)$ ``knows" some very important information about $F$. We will show that  similar results hold  for group algebras of limit groups.  
\end{abstract}

\tableofcontents
\section{Introduction}

This is the third paper (after \cite{assoc}, \cite{KMeq}) in a series of papers on the project on model theory of algebras outlined in our talk  at the ICM in Seoul \cite{ICM}. 
 
Tarski's problems on groups, rings, and other algebraic structures were very inspirational  and led to some  important developments in modern algebra and model theory. 
Usually solutions to  these   problems for some  structure clarify the most fundamental algebraic properties of the structure and give perspective on the expressive power of the first-order logic in the structure.  We mention here results on first-order theories   of algebraically closed fields,  real closed fields \cite{Tarski2},  fields of $p$-adic  numbers \cite{Ax-Kochen, Ershov}, abelian groups and modules \cite{Szmielewa, Baur}, boolean algebras \cite{Tarski3, Ershov3},  free and hyperbolic groups \cite{KM1, KM2, Sela, Sela8}, free associative algebras \cite{assoc}, group algebras \cite{KMeq}.

In this paper we show  that unlike free groups, two group algebras of free groups over infinite fields are elementarily equivalent if and only if the groups are isomorphic and the fields are equivalent in the weak second order logic (Theorem \ref{th:12}).    We will prove that for a finitely generated free group $F$ and infinite field $K$, if   $L$  is field and $H$ is  a group such that there is an element in $H$ with finitely generated centralizer  in $L(H)$,  the group algebras $K(F)$ and $L(H)$ are elementarily equivalent ($K(F)\equiv L(H)$) if and only if 
 $H$ is isomorphic to $F$  and the fields are equivalent in the weak second order logic (Theorem \ref{th:eeq}).
Notice that for any  field, the Diophantine problem, and, therefore, the theory of a  group algebra of a torsion free hyperbolic group is undecidable  \cite{KMeq}. Notice also that for any  field $K$ the first-order theory of the group algebra $K(F)$ is not stable because $K(F)$ is an integral domain (but not a field). We will prove the definability of bases of $F$ in $K(F)$  when $K$ is an infinite field (Theorem \ref{th:bases}).  In contrast to this, we proved that primitive elements, and, therefore, free bases are not definable in a free group of rank greater than two  \cite{KMdef}. This was a solution of an old Malcev's problem, and the  same question about group algebras is also attributed to Malcev.  We will also show that  the set of geodesics is definable, and many geometric  properties of $F$ are definable in $K(F)$ (Theorem \ref{th:geom}). Hence $K(F)$ ``knows" some very important information about $F$ (this information can be expressed in the first order language) that $F$ itself does not know. We will show that  similar results hold  for group algebras of  limit groups (Theorems \ref{th:geom}, \ref{th:eeq2},   \ref{th:limit}, \ref{th:geom1}). Recall, that a limit group is a finitely generated fully residually free group, namely a group $G$ such that for any finite set of distinct elements in $G$ there is a homomorphism into a free group that is injective on this set.  We will also prove (Theorem \ref{th:subring}) definability of finitely generated subrings, subgroups, ideals and submonoids in $K(G)$  for a non-abelian limit group $G$ and infinite field $K$ (in some cases $K$ has to be  interpretable in $\N$), and prove that the theory of $K(G)$  
does not admit quantifier elimination to boolean combinations of formulas from $\Pi_n$ or $\Sigma_n$ (with constants from $K(G)$) for any bounded $n$ (Theorem \ref{th:elimination}).

Our main tool in proving these is the method of first-order interpretation. 
We remind here some precise definitions and several known facts that may not be very familiar to algebraists. 

A language $L$ is a triple $(\mathcal{F}_L, \mathcal{P}_L, \mathcal{C}_L)$, where $\mathcal{F}_L = \{f, \ldots \}$ is a  set of functional symbols $f$ coming together with their arities  $n_f \in \mathbb{N}$,  $\mathcal{P}_L$ is  a set of relation (or predicate) symbols $\mathcal{P}_L = \{P, \ldots \}$ coming together with their arities  $n_P \in \mathbb{N}$,  and a set of constant symbols $ \mathcal{C}_L = \{c, \ldots\}$. Sometimes we write $f(x_1, \ldots,x_n)$ or $P(x_1, \ldots,x_n)$ to show that $n_f = n$ or $n_P = n$.  Usually we denote variables by small letters $x,y,z, a,b, u,v, \ldots$, while the same symbols with bars $\bar x,  \ldots$ denote tuples of the corresponding variables $\bar x = (x_1, \ldots,x_n), \ldots $. A structure in the language $L$ (an $L$-structure) with the base set $A$ is sometimes denoted by $\mathbb{A} = \langle A; L\rangle$ or simply by 
$\mathbb{A} = \langle A; f, \ldots, P, \ldots,c, \ldots \rangle$.  For a given structure $\mathbb{A}$ by $L(\mathbb{A})$ we denote the language of $\mathbb{A}$. Throughout this paper we use frequently the following languages that we fix now: the language of groups $\{ \cdot, ^{-1}, 1\}$, where $\cdot$ is the binary multiplication symbol, $^{-1}$ is the symbol of inversion, and $1$ is the constant symbol for the identity; and the language of unitary rings $\{+, \cdot, 0, 1\}$ with the standard symbols for addition, multiplication, and the additive identity $0$.  When the language $L$ is clear from the context, we follow the standard algebraic practice and denote the structure $\mathbb{A} = \langle A; L\rangle$ simply by $A$. For example, we refer to a field  $\mathbb{F}  = \langle F; +,\cdot,0,1 \rangle$ simply  as $F$, or to a group $\mathbb{G} = \langle G; \cdot,^{-1},1\rangle$ as $G$, etc.

Let $\mathbb{B} = \langle B ; L(\mathbb{B})\rangle$ be a structure. A subset $A \subseteq B^n$ is called {\em definable} in $\mathbb{B}$ if there is a formula $\phi(x_1, \ldots,x_n)$ in $L(\mathbb{B})$ such that  $A = \{(b_1,\ldots,b_n) \in B^n \mid \mathbb{B} \models \phi(b_1, \ldots,b_n)\}$. In this case one says that $\phi$ defines $A$ in $\mathbb{B}$.  Similarly, an operation $f$ on the subset  $A$ is definable in $\mathbb{B}$ if its graph is definable in $\mathbb{B}$.  An $n$-ary predicate $P(x_1,\ldots ,x_n)$ is definable in $\mathbb{B}$ if the set $\{(b_1,\ldots ,b_n)\in\mathbb{B}^n| P(b_1,\ldots ,b_n) \ { \rm is\  true}\}$
is definable in $\mathbb{B}$.

In the same vein  an algebraic structure $\mathbb{A} = \langle A ;f, \ldots, P, \ldots, c, \ldots\rangle$  is definable in $\mathbb{B}$ if there is a definable subset $A^* \subseteq  B^n$ and operations $f^*, \ldots, $ predicates $P^*, \ldots, $ and constants $c^*, \ldots, $ on $A^*$ all definable in $\mathbb{B}$ such that the structure $\mathbb{A}^* = \langle A^*; f^*, \ldots, P^*, \ldots,c^*, \ldots, \rangle$ is isomorphic to $\mathbb{A}$.
For example, if $Z$ is the center of a group $G$  then it is definable as a group  in $G$, the same for the center of a ring.

\begin{definition} \label{de:interpretable} An algebraic  structure $\mathbb{A} = \langle A ;f, \ldots, P, \ldots, c, \ldots\rangle$  is interpretable  in a structure $\mathbb{B}$  if there is a  subset $A^* \subseteq B^n$  definable in $\mathbb{B}$, an equivalence relation $\sim$ on $A^*$ definable in $\mathbb{B}$, operations  $f^*, \ldots, $ predicates $P^*, \ldots, $ and constants $c^*, \ldots, $ on the quotient set $A^*/{\sim}$ all interpretable in $\mathbb{B}$ such that the structure $\mathbb{A}^* = \langle A^*/{\sim}; f^*, \ldots, P^*, \ldots,c^*, \ldots, \rangle$ is isomorphic to $\mathbb{A}$.
 \end{definition}

Interpretation of $\mathbb{A}$ in a class of structures $\mathcal{C}$ is    {\em uniform}  if the formulas that interpret   $\mathbb{A}$ in a structure $\mathbb{B}$ from $\mathcal{C}$   are the same for every structure $\mathbb{B}$ from $\mathcal{C}$.

Sometimes, to define a subset or interpreted a structure $\mathbb{A}$ in a   given structure $\mathbb{B}$ one has to add some elements, say from a subset $P \subseteq B$  to the language $L = L(\mathbb{B})$ as  new constants (we denote the resulting language by $L(\mathbb{B})_P$).    In this case we say that $\mathbb{A}$ is {\em interpretable   with parameters} $P$  in $\mathbb{B}$.  {\it Uniform interpretability with parameters} in a class   $\mathcal{C}$ means that the formulas that interpret   $\mathbb{A}$ in a structure $\mathbb{B}$ from $\mathcal{C}$   are the same for every structure $\mathbb{B}$ from $\mathcal{C}$ and parameters in each such  $\mathbb{B}$ come from  subsets uniformly definable in $\mathcal{C}$.
If we want to emphasize that the interpretability is without constants we say {\em absolutely} interpretable or {\it $0$-interpretable}. In most cases we have the absolute interpretability, so if not said otherwise, throughout the paper interpretability means absolute  interpretability.  
We write $\MA \to_{int} \MB$ when $\MA$ is absolutely interpretable in $\MB$. The key result about interpretability is the following lemma.
\begin{lemma}\cite{Hodges}\label{Hodges}
If $\mathbb{A}$ is interpretable in $\mathbb{B}$ with parameters $P$,
then for every formula $\psi(\bar{x})$ of $L(\mathbb{A})$, one can
effectively construct a formula $\psi^{*}(\bar{z},P)$ of
$L(\mathbb{B})$ such that for any $\bar{a} \in \mathbb{A}$, one has
that $\mathbb{A} \models \psi(\bar{a})$ if and only if $\mathbb{B}
\models [\psi(\bar{a})]^*$.

\end{lemma}

We show that for a free group of  finite rank $F=F(X)$, the arithmetic  $\N = \langle N, +,\cdot, 0\rangle$, and the weak second order theory of the infinite field $K$  are all interpretable in $K(F)$,  uniformly in $K$ and $X$. Here we say that the weak second order theory of  a structure $B$ is  interpretable in $K(F)$ if the first-order structure $HF(B)$ of hereditary finite sets over $B$, or equivalently, the list superstructure $S(B,\N)$, is interpretable in $K(F)$ (see Section \ref{se:weak-second-order} for precise definitions).  On the other hand,  one can construct  an interpretation $K(F)^\diamond$ of $K(F)$ in $S(K,\N)$. Since $S(K,\N)$ is interpretable in $K(F)$ this gives the interpretation  $K(F)^{\diamond\diamond}$ of $K(F)$ in itself. In fact, one can interpret any "constructive over $K$" algebra $L$ in $S(K,\N)$, but usually this interpretation $L^\diamond$ and the original algebra $L$ are not related much. However, in the case of $K(F)$ we showed that there { are}  definable isomorphisms between $K(F)^{\diamond\diamond}$ and $K(F)$
 and between $S(K,\N)$ and $S(K,\N)^{\diamond\diamond}$ (in model theory this situation is described in terms of {\em bi-interpretability}, see Section \ref{biint}). This bi-interpretability gives a powerful tool to study arbitrary rings which are first-order equivalent to a given algebra $K(F)$ viewed as a ring (Theorem 11).

 In Section 2 we  obtain results about the ring of Laurent polynomials and group algebras of commutative transitive groups and groups satisfying Kaplansky's unit conjecture. In Section 3 we  prove results about algebras elementarily equivalent to a group algebra of a left orderable  hyperbolic group. In Section 4 we discuss unique factorization in $K(F)$ and prove some of the main technical lemmas. In Section 5 we  prove bi-interpretability and most of the main results mentioned above. In Section \ref{elimination}  we study arithmetical hierarchy and elimination of quantifiers in  group rings of limit groups and  prove Theorem \ref{th:elimination}.
 \section{Group Rings}
\label{se:group-rings}

\subsection{Laurent polynomials}

Denote by $K[x,x^{-1}]$ the ring of Laurent polynomials in one variable $x$ over a field $K$.
\begin{lemma}\label{natural}
Let $K$ be an infinite  field. Then the following holds. \begin{enumerate}\item  For a non-invertible element $P \in K[x,x^{-1}]$ which is the sum of at least three monomials,  the polynomial ring $K[P]$ is definable in $K[x,x^{-1}]$ with parameter $P$ uniformly in $K$  and $P$. 
\item For any non-invertible elements $P,Q\in K[x,x^{-1}]$ which are the sums of at least three monomials, the canonical isomorphism of interpretations $\mu _{P,Q}:  K[P]\rightarrow K[Q]$ is definable in $K[x,x^{-1}]$ uniformly in $K,P,Q.$
\item $K[t]$ is $0$-interpretable in $K[x,x^{-1}]$.\end{enumerate} 

\end{lemma} 
\begin{proof}  
 1) Every element in $K[x,x^{-1}]$ can be uniquely written as $(\Sigma\alpha _ix^{n_i})/x^k$, where  $n_i\geq 0$ for each $i$, $k\geq 0$ and minimal possible for this element. 
 We can partially order elements of $K[x,x^{-1}]$  according to the left lexicographic order of  pairs $(k, m),$  where $m=max \{n_i\}.$ If $Q$ is a polynomial, we denote these numbers $k(Q), m(Q).$

By \cite{KMeq}, Lemma 8 (or Lemma \ref{le:units} below),  $K$ is definable in $K[x,x^{-1}]$. Since infinite cyclic group satisfies Kaplansky's unit conjecture (see Section \ref{2.3})  monomials are exactly the units in $K[x,x^{-1}]$. Fix a non-unit element $P \in K[x,x^{-1}]$ that is the sum of at least three monomials. This property is definable, namely $P$ must be a non-unit and not the sum of two units.  
We will show that the following formula  with the parameter $P$ defines the ring of polynomials $K[P]$ in $K[x,x^{-1}]$:
 $$
 \psi(Q,P) = \forall \alpha \in K \exists \beta \in K (P-\alpha \mid Q-\beta).
 $$
 Indeed, any $Q \in K$ satisfies the formula for $\beta = Q$. Suppose now  $Q\in K[P] \smallsetminus K$ then for any $\alpha \in K$, $Q = (P-\alpha)Q_1 +\beta$ for some $\beta \in K$. Hence, $P-\alpha \mid Q-\beta$, so $Q$ satisfies $\psi(Q,P)$ in $K[x,x^{-1}]$. 
 
 On the other hand, if $K[x,x^{-1}]\models \psi(Q,P)$ for some $Q \in K[x,x^{-1}]$, then for a given $\alpha \in K$ one has $Q -\beta = (P-\alpha)Q_0$ for some $\beta \in K$ and $Q_0 \in K[x,x^{-1}]$. For another $\alpha_1\in K$ there exists $\beta_1 \in K$ such that  
 $(P-\alpha_1)\mid Q-\beta_1$. Now, 
 $$
 Q-\beta = (P-\alpha)Q_0 =  (P-\alpha_1 +\alpha_1- \alpha) Q_0= 
 (P-\alpha_1)Q_0 +(\alpha_1-\alpha)Q_0.
 $$
  Hence 
  $$Q-\beta_1 = Q-\beta +\beta -\beta_1 =  (P-\alpha_1)Q_0 +(\alpha_1-\alpha)Q_0 + \beta -\beta_1. 
   $$
    It follows that $P-\alpha_1 \mid (\alpha_1-\alpha)Q_0 + \beta -\beta_1$, and  $P-\alpha_1 \mid Q_0 + (\beta -\beta_1) (\alpha_1-\alpha)^{-1}$, therefore  $K[x,x^{-1}] \models \psi(Q_0,P)$.  We need this below to apply induction.

    Represent $P=\bar Px^{-k(P)},\ Q=\bar Qx^{-k(Q)}, Q_0=\bar Q_0x^{-k(Q_0)}.$ Then we have 
    
    $$\bar Qx^{-k(Q)} -\beta = (Px^{-k(P)}-\alpha)Q_0x^{-k(Q_0)}. $$
    This equality implies the following: \begin{itemize} \item If $k(Q)=0$, then $k(Q_0)=0$ and $k(P)=0$, 
\item   If $k(Q)>0$,  $k(Q_0)>0$ and $k(P)>0$, then $k(Q)=k(P)+k(Q_0).$
\end{itemize}    
    We will now prove by  induction on $k(Q)$ that  if $K[x,x^{-1}]\models \psi(Q,P)$, then $Q\in K[P]$.
    
    The basis of induction is the case $k(Q)=0.$ In this case $Q,P$ and $Q_0$ belong to $K[x]$, and we can use the argument of \cite{assoc}, Lemma 6. Namely
     that the leading term in $Q_0$ is smaller  than that one in $Q$.  Hence, by induction on $m(Q)$ in the case $k(Q)=0$,  $Q_0$ belongs to $K[P]$, so  does $Q$.
     
     Suppose now that $k(Q)>0$,  $k(P)>0$.  If $k(Q_0)=0$  or $k(Q_0)>0$, then $k(Q_0)<k(Q)$, and by  induction, $Q_0$ belongs to $K[P]$, so  does $Q$. 
     
     The last case to consider is $k(Q)>0$, $k(P)=0$,  so $P\in K[x]$ and has degree at least two as the sum of at least three monomials.   We will show that this case is impossible. In this case for any $\alpha \in K$, $k(Q_0)>0$  and we denote $k=k(Q)=k(Q_0).$ We have $(P-\alpha)\bar Q_0x^{-k}=\bar Qx^{-k}-\beta ,$
     and both $\bar Q, \bar Q_0$ have non-trivial constant term. We rewrite this equation as 
     $$(P-\alpha)\bar Q_0=\bar Q-\beta x^k.$$ If $m(Q)\leq k$, then $m(Q_0)\leq k-2$, because $P$  has degree at least two. Since $Q$ satisfies the formula, we have to be able to find, given $P$ and $\bar Q$, for any $\alpha \in K$, coefficients $y_0,\ldots ,y_t$, $t\leq k-2$, of  $\bar Q_0$ and $\beta \in K.$  This gives a system of $k+1$ linear equations with at most $k$ variables ($t+2$ variables) that should have a solution for any $\alpha$.  
 If we write equations for coefficients of the monomials $x^t,\ldots ,x, 1$ respectively, we obtain a triangular system of equations in variables $y_t,\ldots ,y_0$ with elements $(p_0-\alpha)$ on the diagonal (here $p_0$ is the constant term of $P$).    
Since $K$ is infinite, one can choose $\alpha$ such that the system does not satisfy the equation for the coefficients of $x^{t+1}$ and, therefore, does not have  a solution.  We can now apply induction on $m(Q)$ for Laurent polynomials $Q$ with fixed $k=k(Q)$.  If $Q$ satisfies the formula, and $k(Q)>0$, then $m(Q)>k(Q)$ and   $m(Q_0)<m(Q)$, moreover, $Q_0$ satisfies the formula. But $Q_0$ cannot satisfy the formula by  induction because $k(Q_0)=k(Q)$. Therefore, by contradiction, $Q$ cannot satisfy the formula, and the case $k(Q)>0$, $k(P)=0$ is impossible.
     
2)  Let $f(P)\in K[P]$ and $g(Q)\in K[Q]$. Since the field $K$ is infinite,
$g(Q)=\mu _{P,Q}(f(P))$ if and only if $$\forall \alpha \in K-\{0\}\forall \beta \in K-\{0\} ((P-\alpha)|(f(P)-\beta)\iff (Q-\alpha)|(g(Q)-\beta )).$$ Indeed, this formula states that the values of $f(P)$ and $f(Q)$  are the same for any $\alpha \in K-\{0\}.$

3) follows from 2). Namely, using the isomorphism $\mu _{P,Q}$ one can glue all interpretations $K[P]$ into one equivalence class isomorphic to $K[t]$. Therefore $K[t]$ is $0$-interpretable in $K[x,x^{-1}]$.
\end{proof}

\begin{lemma} \label{le:limit1} Let $R=K[x_1,x_1^{-1},\ldots ,x_n,x_n^{-1}]$ be a ring of Laurent polynomials over an infinite field, $g=x_1^{\gamma _1}\ldots x_n^{\gamma _n}$, where $\gamma _1,\ldots ,\gamma _n\in \Z$,  a  non-trivial monomial, that is not a proper power, then $K[g,g^{-1}]$ is definable in $R$.
\end{lemma}
\begin{proof} Let $A=\{g^{n}, n\in\Z\}$. Then $A$ is definable as  the set of invertible elements $x$ in $R$ and not in $K$ such that $(g-1)|(x-1)$.
The formula
 $$
 \psi(Q,g) = \exists a\in A\forall \alpha \in K \exists \beta \in K \ (g-\alpha \mid Q-\beta a).
 $$
defines $K[g,g^{-1}]$ in $R$. Indeed, if $Q\in K[g,g^{-1}]$, then for some $k$, $g^kQ\in K[g]$ and $g^kQ=(g-1)Q_0+\beta$, where $\beta\in K$. Then we take $a= g^{-k}.$
 

We order  elements of $R$ the following way. Let $q\in R$, and let $y$ be a monomial in $K[x_1,\ldots ,x_n]$ of minimal degree
such that $qy\in K[x_1,\ldots ,x_n]$. Then $q_1\leq q_2$ if $(y_1,y_1q)\leq (y,q),$ where pairs of polynomials $(y,qy)$ ordered left-lexicographically, and arbitrary monomial order used in $K[x_1,\ldots ,x_n]$. Pairs corresponding to polynomials in $K[x_1,\ldots ,x_n]$ have  first component of degree zero.
 
 On the other hand, without loss of generality we can assume $g=x_1^{\gamma _1}\ldots x_n^{\gamma _n}$, where $\gamma _1,\ldots ,\gamma _n\in\  N$. if $R\models \psi(Q,g)$ for some $Q \in R$, then there exists $a=g^{-k}$ such that for a given $\alpha \in K$ one has $Q -\beta a= (g-\alpha)Q_0$ for some $\beta \in K$ and $Q_0 \in R$. 
 Notice that we need the same minimal degree monomial $y$,  so that $yQ$ and $yQ_0$ are polynomials, and the degree of $Q_0$ is less than degree of $Q$. Therefore $Q$  is greater than $Q_0$ in our order.

 For another $\alpha_1\in K$ there exists $\beta_1 \in K$ such that  
 $(g-\alpha_1)\mid Q-\beta_1a$. Now, 
 $$
 Q-\beta a = (g-\alpha)Q_0 =  (g-\alpha_1 +\alpha_1- \alpha) Q_0= 
 (g-\alpha_1)Q_0 +(\alpha_1-\alpha)Q_0.
 $$
  Hence 
  $$Q-\beta_1 a = Q-\beta a+\beta a-\beta_1 a=  (g-\alpha_1)Q_0 +(\alpha_1-\alpha)Q_0 + (\beta -\beta_1)a 
   $$
    It follows that $g-\alpha_1 \mid (\alpha_1-\alpha)Q_0 + (\beta -\beta_1)a$, and  $g-\alpha_1 \mid Q_0 + (\beta -\beta_1) (\alpha_1-\alpha)^{-1}a$, therefore  $F[X] \models \psi(Q_0,g)$.  Hence, by induction, $Q_0$ belongs to $K[g,g^{-1}]$, so  does $Q$.

\end{proof}

\subsection{Group rings of commutative transitive torsion-free groups}\label{se:1}

Recall that a group is commutative-transitive (or CT) if the commutation relation  is transitive on non-trivial elements of $G$. Another way to describe this property is to say that centralizers of non-trivial elements in $G$ are abelian. 
There are many classes of CT groups: abelian, free, torsion-free hyperbolic \cite{G}, limit groups \cite{BB}, total relatively hyperbolic (because they are CSA by \cite{O}), solvable Baumslag-Solitar groups \cite{Wu}, free solvable groups \cite{Wu}, etc. We refer to papers \cite{F,Wu} on general theory of CT-groups.

\begin{lemma} \label{pr:centralizer-hyp} (\cite{KMeq}, Lemma 9) Let $G$ be a  commutative transitive torsion-free group and $K$ a field.  Then for every $g\in G$ the centralizer $C_{K(G)}(g)$  of $g$ in $K(G)$ is isomorphic to the group algebra $K(C_G(g))$ of the centralizer $C_G(g)$ of $g$ in $G$. In particular, if $C_G(g)$ is free abelian  with basis $\{t_i, i\in I\}$ then the centralizer $C_{K(G)}(g)$ is isomorphic to the ring of the Laurent  polynomials $K[t_i, t_i^{-1}, i\in I]$.
\end{lemma} 

\begin{cor} 
If $G$ is a torsion-free hyperbolic group then for any non-trivial $g\in G$ the centralizer  $C_{K(G)}(x)$ is isomorphic to the ring of Laurent polynomials in one variable $K[t,t^{-1}]$, where $t$ is a generator of the infinite cyclic group $C_G(g)$ - the centralizer of $g$ in $G$.
\end{cor}

\begin{theorem} \label{th:undec-group-rings} (\cite{KMeq}, Theorem 8) Let $K$ be a field,  $G$ a   torsion-free group, and $g \in G$ such that 
\begin{itemize}
\item the centralizer $C_G(g)$ is a countable free abelian  group;
\item   $C_G(g^k) = C_G(g)$ for any non-zero integer $k$. 
\end{itemize}
 Then  the arithmetic ${\mathbb N}$ is  equationally interpretable in the group algebra $K(G)$,   the Diophantine problem and, therefore,  the first-order theory $Th(K(G))$ of  $K(G)$ is undecidable. 
\end{theorem}

There are many groups that satisfy the conditions of Theorem \ref{th:undec-group-rings}:
free, torsion-free hyperbolic, toral relatively hyperbolic, free solvable, wreath products of  abelian  torsion-free groups by torsion free CT-groups (in this case centralizers are abelian),  etc.

\subsection{Groups satisfying Kaplansky's unit conjecture}\label{2.3}

For a ring $R$ denote by $R^\ast$ the set of units in $R$. Recall that a  group $G$ satisfies Kaplansky's unit conjecture  if for any field $K$  units in the group algebra $K(G)$ are only the obvious ones $\alpha\cdot g$, where $\alpha \in K \smallsetminus \{0\}$ and $g \in G$.  In our case, when $G$ is torsion-free, Kaplansky's unit conjecture implies that there are no zero divisors in $K(G)$.

\begin{lemma}(\cite{KMeq}, Lemma 8)\label{le:units}
Let $G$ be a a torsion-free group satisfying Kaplansky's unit conjecture. Then for any field $K$ the following hold:
\begin{itemize}
\item [1)] the field $K$ is a Diophantine subset of $K(G)$ uniformly for any field $K$;
\item [2)] the group $G$ is Diophantine interpretable in $K(G)$  uniformly for any group $G$.
\end{itemize}
\end{lemma}

A group $G$ is called {\em left orderable}  (LO) if there is a linear ordering on $G$ which respects the left multiplication in $G$, i.e, for any $x,y,z  \in G$ if $x \leq y$ then $zx \leq zy$. Notice that LO group is torsion-free.

There are many LO groups: free groups, limit groups, right-angled Artin
groups and their subgroups \cite{Dub}, hence all special groups (actually,
they are even bi-orderable), finitely presented groups acting freely on $%
\Lambda$-trees \cite{KMS09}, free solvable groups (bi-orderable), locally
indicapable \cite{Agol}, in particular, torsion-free one relator groups \cite%
{Howie}, braid groups \cite{Deh}. Not every torsion-free
hyperbolic group is LO (\cite{BGW}, Theorem 8).   Every limit group  is a
subgroup of a non-standard free group (that is an ultrapower of a free group)  and, therefore, is bi-orderable
(see for instance in \cite{KMS09}).

For us the crucial fact on LO groups is that every LO group satisfies the Kaplansky's unit and zero-divisors conjectures.  This follows, for example, from \cite{P} and \cite{ST}.  Indeed LO groups have the so-called ``unique product property" by \cite{ST}, and groups with this property satisfy the Kaplansky's unit and zero-divisors conjectures \cite{P}.

\begin{lemma} 
\label{le:equivLO} 
If $G$ is  LO, and $H\equiv G$, then $H$ is LO.
\end{lemma}
\begin{proof}
The group $G$ is LO, hence one can view $G$ in the extended language $L =\{\cdot, \leq, 1\}$. The extended structure satisfies the LO axiom which can be written by a first-order sentence in $L$. Hence the ultrapower $ G^\N/\omega$ of $G$ over a non-principal $\omega$ is also LO, since the $L$-structures  $G$ and $ G^\N/\omega$  are elementarily equivalent. If $G \equiv H$ in the group language then by Keisler-Shelah isomorphism theorem their ultrapowers $G^\N/\omega$ and $H^\N/\omega$ are isomorphic  in the group language, so there is a group isomorphism, say $\theta: G^\N/\omega$ to $H^\N/\omega$. The image under $\theta$ of the ordering $\leq$ gives an ordering, say  $\leq'$ on  the set $H^\N/\omega$. Because $\theta$ is also a  group isomorphism the ordering $\leq'$ is compatible with the left multiplication in $H^\N/\omega$, so the group  $H^\N/\omega$ is LO.   It follows that $H$ as a subgroup of $H^\N/\omega$ is also LO, as claimed.
\end{proof}

 To prove the main result of this section we need another lemma.
 
 \begin{lemma} \label{le:1+h}
 Let $H$ be a torsion-free group and $L$ a field. Then for any $1 \neq h \in H$ the element $1-h$ is not invertible in the group algebra $L(H)$.
 \end{lemma}
\begin{proof}
Let $1 \neq h \in H$. If $1-h$ is invertible in $L(H)$ then there exists an element $f = \sum_{i = 1}^k \alpha_i f_i \in L(H)$, where $\alpha_i \in L$ and $f_i $ are distinct elements of $H$, such that $(1-h) f = 1$. Hence
$$
\sum_{i = 1}^k \alpha_i f_i  - \sum_{i = 1}^k \alpha_i hf_i  = 1.
$$
It follows that one  of the elements $f_i$ or $hf_i$, say $hf_{i_0}$,  is equal to 1. Then denoting $\beta_i = \alpha_i$ for $i \neq i_0$; and $\beta_{i_0} = 1+\alpha_{i_0}$, one gets
$$
\sum_{i = 1}^k \alpha_i f_i  = \sum_{i = 1}^k \beta_i hf_i.
$$
Since all elements $f_i$, as well as $hf_i$, are distinct it follows that there is a permutation $\sigma \in Sym(k)$ such that $f_i = hf_{\sigma(i)}$ for every $i$. Let $\sigma = \sigma_1 \ldots \sigma_m$ be a decomposition of $\sigma$ into a product of independent cycles. If one of these cycles has length 1 then $f_i = hf_i$ for some $i$, which implies that $h = 1$ - contradiction with the choice of $h$. Hence all the cycles $\sigma_i$ have length at least 2. Consider the cycle $\sigma_1$, to simplify notation we may assume (upon rearrangement of indices) that $\sigma_1 = (1 2 \ldots d)$, where $2\leq d \leq k$.  Hence
$$
f_1 = hf_2, \ldots, f_d = hf_1,
$$
so $f_1 = h^df_1$, which implies that $h^d = 1$. Since $H$ is torsion-free one has $h = 1$ - contradiction. This shows that $1-h$ is non-invertible in $L(H)$.
 
\end{proof}

 \begin{theorem} \label{th:lo} Let $G, H$ be groups and $K, L$ be fields such that $K(G)\equiv L(H)$. If $G$ is LO then the following hold:
 \begin{itemize}
 \item [1)] $H$ is LO.
 \item [2)] $K \equiv L$  and $G \equiv H$.
 \end{itemize}
 \end{theorem}
 \begin{proof}  To prove 1) consider the formula 
 $$
\phi(x) = (x = -1) \vee (x \in K(G)^*
) \wedge ( (x+1) \in K(G)^*),
$$
 which defines $K^{*}$ in $K(G)^*$ (as well as in $K(G)$). The same formula defines some subgroup $U$ in the group of units $L(H)^*$.  Since $K^*$ is central in $K(G)^*$ the subgroup $U$ is also central in $L(H)^*$, in particular, it is normal in $L(H)^*$. The formula $\phi(x)$ states that the elements in $U$ are precisely the units $a \in L(H)^*$ for which either $1+a $ is invertible in $L(H)$ or $a = -1$.  Notice that  the ring $K(G)$ has no zero-divisors. Since this property can be described by a first-order sentence the ring $L(H)$ which elementarily equivalent to $K(G)$ also does not have zero-divisors. This implies that the group $H$ is torsion-free, otherwise, if $h^n = 1$ for some $1 \neq h \in H$ then $0 = h^n - 1 = (h-1)(h^{n-1} +h^{n-2} + \ldots +h+1)$ and $L(H)$ would have zero-divisors.  Now by Lemma  \ref{le:1+h} $H \cap U = 1$. Hence $H$ embeds into the quotient group $L(H)^*/U$. Notice that this group $L(H)^*/U$ is interpretable in $L(H)$ precisely by the same formulas as $G$ in $K(G)$ (see the proof of Lemma \ref{le:units}). Hence $G \equiv L(H)^*/U$.  The group $G$ is LO, hence by Lemma \ref{le:equivLO} the group $L(H)^*/U$ is LO, as well as its subgroup $H$. This proves 1).
 
 Now 2) follows from 1) since by Lemma \ref{le:units} the fields $K$ and $L$, as well as the group $G$ and $H$,  are interpretable in the rings $K(G)$ and $L(H)$ by the same first-order formulas, hence they are elementarily equivalent.
  \end{proof}
  \begin{remark}   For any limit group  $G$ and any  field $K$ the first-order theory of the group algebra $K(G)$ is not stable.  
 \end{remark}
Indeed, the group algebra $K(G)$ is an integral domain (but not a field).

 \section{Weak second order logic and elementary equivalence of group algebras}
 
\subsection{Weak second order logic}\label{se:weak-second-order} 
\label{se:weak-second-order-intro}

For a set $A$ let  $Pf(A)$ be the set of all finite subsets of $A$.
Now we define by induction  the set $HF(A)$  of hereditary finite sets over $A$;
\begin{itemize}
\item  $HF_0(A)= A $,  
\item $HF_{n+1}(A) = HF_n(A)\cup Pf(HF_n(A))$, 
\item  $HF(A)=\bigcup _{n\in\omega}HF_n(A).$  
 \end{itemize}

For a structure $\MA = \langle A;L \rangle$ define a new 
 first-order structure  as follows. Firstly, one replaces all operations in $L$ by the corresponding predicates (the graphs of the operations) on $A$, so one may assume from the beginning that $L$ consists only of predicate symbols. Secondly, consider the structure $$HF(\MA)=\langle HF(A); L, P_A, \in\rangle, $$where $L$ is defined on the subset $A$, $P_A$ defines $A$ in $HF(A)$, and $\in$ is the membership predicate on $HF(A)$.  Then everything that can be expressed in the weak second order logic in $\MA$ can be expressed in the first-order logic in $HF(\MA)$, and vice versa. The structure  $HF(\MA)$ appears naturally in the  weak second order logic, the theory of admissible sets, and $\Sigma$-definability, - we refer to  \cite{B1,B2,E,Ershov2} for details.

 There is another structure, termed the {\em list superstructure} $S(\MA,\MN)$ over $\MA$ whose  first-order theory has the same expressive power as the weak second order logic over $\MA$ and which  is more convenient for us to use in this paper. To introduce $S(\MA,\MN)$ we need a few definitions.  Let $S(A)$ be the set of all finite sequences (tuples) of elements from $A$.
 For a  structure $\MA = \langle A;L \rangle$ define in the notation above  a new two-sorted structure $S(\MA)$ as follows:
$$
S(\MA) = \langle \MA, S(A); \frown, \in\rangle,
$$
where $\frown$ is the binary operation of  concatenation of two sequences from $S(A)$ and $a \in s$ for $a \in A, s \in S(A)$ is interpreted as $a$ being a component of the tuple $s$.  As customary in the formal language theory we will denote the concatenation  $s\frown t$ of two sequences $s$ and $t$ by   $st$.
 
 Now, the structure $S(\MA,\MN)$ is defined as  the three-sorted structure 
 $$
 S(\MA,\MN) = \langle \MA, S(A),\MN; t(s,i,a), l(s), \frown, \in \rangle,
 $$
 where $\N = \langle N, +,\cdot, 0,1\rangle$  is the standard arithmetic, $l:S(A) \to N$ is the length function, i.e., $l(s)$ is the length $n$ of a sequence $s= (s_1, \ldots,s_{n})\in S(A)$, and $t(x,y,z)$ is a predicate on $S(A)\times N \times  A$ such that $t(s,i,a)$ holds in $S(\MA,\MN)$ if and only if $s = (s_1, \ldots,s_{n})\in S(A), i \in N, 1\leq i \leq n$, and $a = s_i \in A$. Observe, that in this case the predicate $\in$ is 0-definable in $S(\MA,\N)$ (with the use of $t(s,i,a)$), so sometimes we omit it from the language.
 
 \begin{lemma} \label{th:HF-group1}
Let $K$ be an infinite field. Then  $S(K,\N)$ is  $0$-interpretable  
 in $K[x_1,{x_1}^{-1}, \ldots ,x_n,{x_n}^{-1}]$ uniformly in  $K$ and $n$.\end{lemma}

\begin{proof} 
Follows from (\cite{assoc}, Theorem 5), Lemmas \ref{natural} and \ref{le:limit1}.
  \end{proof}

  In the following lemma we summarize some known results (see for example \cite{bauval}) about the structures $HF(\MA), S(\MA)$, and $S(\MA,\N)$.
  
  \begin{lemma}
  Let $\MA$ be a structure. Then the following holds:
    $$S(\MA) \to_{int} S(\MA,\N) \to_{int} HF(\MA) \to_{int} S(\MA,\N) \to_{int} S(\MA)$$ 
    uniformly in $\MA$ (the last interpretation requires that $\MA$ has at least two elements).
   \end{lemma}
   \begin{cor} \label{c2} $S(K(G),\N)\to_{int} S(S(K(G),\N),\N)  \to_{int}S(K(G),\N).$
   \end{cor}
  
  The following result is known, it is based on two facts: the first one is that there are effective enumerations (codings) of the set of all tuples of natural numbers such that the natural operations over the tuples are computable on their codes;  and the second one is that  all computably enumerable predicates over natural numbers  are 0-definable in $\N$ (see, for example, \cite{Coper, Rogers}).  
  
\begin{lemma} \label{le:list-superstructure}
The list superstructures  $S(\N,\N)$ and $S(\Z,\N)$ are absolutely interpretable in $\N$.  They are also bi-interpretable with $\N$ (see Definition \ref{bi-in} below).
\end{lemma}

 \begin{theorem} \label{th:HF-group}  Let $K$ be an infinite field. 

1) The structure $S(K,\N)$ and the ring of polynomials in finitely many variables $K[X]$  are  mutually interpretable in one another uniformly in $K$ and $X$.

2) The structure $S(K,\N)$ and the ring of Laurent polynomials in finitely many variables $K[X, X^{-1}]$  are mutually interpretable in one another uniformly in $K$ and $X$.
\end{theorem}

\begin{proof} 1) 
It was shown in \cite{assoc}, Theorem 5 that $S(K,\N)$ is interpretable in  $K[X]$. We will recall this proof here.  For a  non-invertible polynomial $P \in K[X]$ the polynomial ring $K[P]$ is definable in $K[X]$ with parameter $P$ uniformly in $K, X$ and $P$. The structure $S(K,\N)$ is interpretable in a ring of polynomials in one variable, say $K[t]$,  with the variable $t$ in the language, uniformly in $K$. To this end consider the language of ring theory $L_t$ with the element $t$  as a new constant.   By  \cite{assoc}, Lemma 9 the arithmetic $\N_t$ is interpretable in $K[t]$ in the language $L_t$ uniformly in $K$. So the set $N_t = \{t^n \mid n \in \N\}$, as well as the addition and the multiplication in $\N_t$, is definable in $K[t]$ by a formula with the parameter $t$.  This gives a required interpretation in $K[t]$ of the third sort $\N$ of the structure
 $$
 S(K,\MN) = \langle K, S(F),\MN; t(s,i,a), l(s),\frown\in \rangle.
 $$
Now we interpret $S(K)$ in $K[t]$. We associate a sequence $\bar \alpha = (\alpha_0, \ldots, \alpha_{n})$ of elements from $K$ with a pair $s_{\bar \alpha} = (\Sigma_{i=0}^n \alpha_it^i,t^n)$. It is shown in \cite{assoc}, Theorem 5  that the set of such pairs is definable in $K[t]$ by a formula in $L_t$.  This gives a 0-interpretation  in $K[t]$ (viewed in the language $L_t$) of the set  $S(K)$ of  all tuples of $K$. Note that the field $K$ is also 0-interpretable in $K[t]$, so the two sorts of the structure $S(K) = \langle K, S(K), \frown,\in\rangle$ are 0-interpretable in $K[t]$ in the language $L_t$. 
It was shown that the operations are also interpretable.

Therefore for a given non-invertible polynomial $P \in K[X]$ one can interpret $S(K,\N)$ in $K[X]$ using the parameter $P$ uniformly in $K$, $X$, and $P$. We denote this interpretation by 
  $$
  S(K,\N)_P = \langle K, S(F)_P, \N_P, t_P(s,i,a), l_P(s), \in _P \rangle .
  $$

For different non-invertible parameters $P_1, P_2 \in K[X]$ there is a uniformly definable isomorphism 
 $$
 \nu_{P_1,P_2}:  S(K,\N)_{P_1} \to S(K,\N)_{P_2}.
 $$
    Observe that the interpretation of the first sort $K$ in $S(K,\N)_P$ does not depend on $P$. The definable isomorphism $\mu_{P_1,P_2}:\N_{P_1} \to \N_{P_2}$ between the third  sorts in $S(K,\N)_{P_1} $ and $S(K,\N)_{P_2} $ was constructed in \cite{assoc}, Lemma 10.     
   The isomorphism $\sigma_{P_1,P_2}: S(K)_{P_1} \to S(K)_{P_2}$ between the second  sorts $S(K)_{P_1}$ and $S(K)_{P_2}$ in $S(K,\N)_{P_1} $ and $S(K,\N)_{P_2} $ which arises  from the identical map $S(K) \to S(K)$ is  definable in $K[X]$ uniformly in $K,X$, $P_1$, and $P_2$. Indeed,  if $s_{\bar \alpha} = (f,P_1^n) \in S(F)_{P_1}$ and $s_{\bar \beta} = (g,P_2^m) \in S(F)_{P_2}$ then for such $\sigma_{P_1,P_2}$ one has $\sigma_{P_1,P_2}(f,P_1^n) = (g,P_2^m) $ if and only if $n = m$ and  the tuples $\bar \alpha$ and $\bar \beta$ are equal. The latter  means that for each  $a,b \in K$ such that $t_{P_1}(s_{\bar \alpha},i,a)$ and $t_{P_2}(s_{\bar \beta},i,b)$ hold in $K[X]$ one has $a = b$.  All these conditions can be written by formulas of the ring theory uniformly in $K, X, P_1, P_2$.  
   
   Using the constructed definable isomorphisms one can glue all
the structures  $S(K,\N)_P$ for different $P$ into one structure isomorphic to $S(K,\N)$.  Therefore $S(K,\N)$  
 is 0-interpretable in K[X] uniformly in K and X.

 On the other hand, if $|X|=r$, a finitely generated free  abelian monoid $M=\langle X\rangle $ is interpretable in $S(\N,\N)\to_{int} \N$. Indeed, a monomial $x_1^{m_1}\ldots x_r^{m_r}$ is represented as a tuple $(m_1,\ldots ,m_r)$.  Similarly,  a polynomial
 $f=\sum _{i=1}^k \alpha_ix_1^{m_{1i}}\ldots x_r^{m_{ri}}$, where for $i<j$, $(m_{1i},\ldots ,m_{ri})< (m_{1j},\ldots ,m_{rj})$ in  left-lexicographic order, can be represented as a tuple
 $$(\alpha _1,m_{11},\ldots ,m_{r1},\alpha _2, m_{12},\ldots ,m_{r2},\ldots )$$  
 in $S(S({K}, \N),\N)$.  It is easy to see that addition and multiplication of polynomials is also definable in $S(S(K,\N),\N)$ (the graph of multiplication is computable and, therefore, definable in $\N$ \cite{Mat}), and, therefore  $K[X]\to_{int} S(S(K,\N),\N)\to_{int} S({K}, \N)$ by Corollary \ref{c2}. So  $K[X]$ is interpretable in $S({K}, \N)$ as $K[X]^{\diamond}$  and, by transitivity, in itself as  $K[X]^{\diamond\diamond}.$ 


2) By Lemma \ref{th:HF-group1}, $S(K,\N)$ is interpretable in the algebra of Laurent polynomials.  We can interpret  $K[X, X^{-1}]$  in $S(K,\N)$ in a similar way as we did for polynomials.
\end{proof}

\begin{lemma} \label{th:HF-group2}  In the situation of Theorem \ref{th:HF-group} the interpretations are such that: 

1) If $M=\langle X\rangle$ a free abelian monoid generated by $X$, $M^{\diamond}$ its image in $S(K,\N)$ and $M^{\diamond\diamond}$ the image of $M^{\diamond}$ when $S(K,\N)$ is interpreted in $K[X]$, then the isomorphism $M\rightarrow M^{\diamond\diamond}$ is definable in $K[X].$

2) If $M$ is a free abelian group generated by $X$ then the isomorphism $M\rightarrow M^{\diamond\diamond}$ is definable in $K[X,X^{-1}].$
\end{lemma}
\begin{proof}
1) The isomorphism $M\rightarrow M^{\diamond\diamond}$ is definable in $K[X]$ because one can  write a formula $\phi (t,f, P,X)$, where $P$ and $X$ are constants ($X$ defines the language), that defines all the pairs $(t,f)$, where $t$ is a tuple representing a monomial $f$. 
2) Similar to 1).
\end{proof}
 
\subsection{Elementarily equivalent group algebras}
The following theorem was proved in \cite{assoc}.
\begin{theorem} (\cite{assoc}, Theorem 4)
Let $G$ be a torsion free non-abelian hyperbolic group. Then the field $K$ and its action on $K(G)$ are interpretable in $K(G)$ uniformly in $K$ and $G$.
\end{theorem}

 We will now prove the following results. 
 \begin{theorem} \label{th:hyp}Let $G$ be LO and  hyperbolic and $H$ a group such that there is an element in $H$ with a finitely generated centralizer. Then for any infinite fields $K$ and $L$, if $K(G) \equiv L(H)$ then 
 \begin{itemize}
 \item  $G \equiv H$ 
 \item $HF(K) \equiv HF(L)$.
 \end{itemize}
\end{theorem}

\begin{theorem} \label{th:sol} Let $G$ be a finitely generated free solvable  group and $H$ a group such that there is an element in $H$ with a finitely generated centralizer. Then for any infinite fields $K$ and $L$ if $K(G) \equiv L(H)$ then 
 \begin{itemize}
 \item  $G \equiv H$ 
 \item $HF(K) \equiv HF(L)$.
 \end{itemize}
\end{theorem}

{\bf Proof of Theorem \ref{th:hyp}.} First we will prove the following proposition.

 \begin{prop} \label{th:S(F,N)-non-comm} 
 Let  $K$ be an infinite field and $G$ be LO and non-elementary  hyperbolic. Then the following hold:
 \begin{itemize} \item [1)] for a given non-invertible  polynomial $P(g)\in K(G),\ g\in G$ which is the sum of at least three monomials one can interpret $S(K,\N)$ in $K(G)$ by $S(K,\N)_P$ above, using the parameter $P$ uniformly in $K, G,$ and $P$. 
 \item [2)] for any two such polynomials $P, Q \in K(G) $ the canonical (unique)  isomorphism of interpretations $\nu_{P,Q} : S(K,\N)_P \to S(K,\N)_Q$ is definable in $K(G) $ uniformly in $K$, $G$, $P$, and $Q$.
\item [3)]  $S(K,\N)$ is 0-interpretable  in $K(G) $  uniformly in $K$ and $G$.
 
 \end{itemize}
 \end{prop}
\begin{proof}

1) The structure $S(K,{\mathbb N})$ for an infinite field $K$ is interpretable in $K[x,x^{-1}]$ using a non-invertible polynomial $P\in K[x,x^{-1}]$ by Theorem \ref{th:HF-group}.  The algebra of Laurent polynomials   $K[x,x^{-1}]$ is definable in $K(G)$ as the  centralizer of an element in $G$. 

To prove 2) observe that since the arithmetic is interpretable, for any such $P$ and $Q$ there is a formula that defines the set of pairs $R = \{(P^m,Q^m) \mid m \in \N\}$ uniformly in $K$, $G$,  $P$, and $Q$ (see \cite{assoc} Lemma 12 (for char 0) and Lemma 17).
Recall that a sequence $s = (\alpha_0, \ldots,  \alpha_m)$ is interpreted in $S(K,\N)_P$ as a pair $s_P = (\sum_{i = 0}^m \alpha_iP^i, P^m) \in S(K)_P$, in $S(K,\N)_P$  and similarly, by the pair $s_Q = (\sum_{i = 0}^m \alpha_iQ^i, Q^m) \in S(K)_Q$ in $S(K,\N)_Q$.  We need to show that the set of pairs $\{(s_P,s_Q) \mid s \in S(K)\}$ is definable in $K(G)$ uniformly in $K, G, P, Q$. Since the set of pairs $R$ is definable it follows that the set of pairs  
$(s_P,r_Q)$ such that $s,r \in S(K)$ and $l_P(s_P) = l_Q(r_Q)$ (i.e., the lengths of the tuples $s$ and $r$ are equal)  is  definable in $K(G)$ uniformly in $K, G, P, Q$.  Recall that the predicates $t_P(s,i,a)$ define in $K(G)$ the coordinate functions $ s_P  \to a  \in K$, where $a$ is the $i$'s term of  the sequence $s_P$,  uniformly in $K, i, G,P$ (here $0 \leq i \leq l(s)$ and $K$ is viewed as the set of invertible elements $k$ in  $K(G)$ such that $k+1$ is also invertible). Therefore, there is a formula which states that for any $0 \leq i \leq l(s) = l(r)$ the sequences   $s_P$ and $r_Q$ have the same $i$ terms. Hence   
the set of pairs 
$$
\{ (\sum_{i = 0}^m \alpha_iP^i, \sum_{i = 0}^m \alpha_iQ^i) \mid \alpha_i \in K, m \in \N\}
$$
 is also definable in $K(G)$ uniformly in $K, G, P$ and $Q$. This gives an isomorphism $ S(K,\N)_P \to S(K,\N)_Q$ definable in $K(G)$ uniformly in $K$, $G$, $P$, and $Q$, as claimed.

 This completes the definition of  the isomorphism of interpretations $\nu_{P,Q} : S(K,\N)_P \to S(K,\N)_Q$, as claimed.

3) follows from 2).
\end{proof}
Theorem \ref{th:hyp} follows from the Proposition.  Theorem \ref{th:sol} can be proved similarly.

\section{Unique factorization in K(F)}

We recall that a finitely generated free group is bi-orderable. Two right ideals $I, I_1$ are called similar in a ring $R$ if the $R$-modules $R/I$ and  $R/I_1$  are isomorphic \cite{Cohn3}. Any non-unit which cannot be written as a product of two non-units is said to be {\em irreducible} or an {\em atom.} An integral domain is said to be atomic if every element which is neither  zero nor a unit, is a product of atoms. Cohn defines a non-commutative UFD (unique factorization domain) as an integral domain which is atomic and is such that any two complete factorizations (into atoms) of an element $x$ are isomorphic, in the sense that they have the same length, say $x=x_1\ldots x_r=y_1\ldots y_r$ and there is a permutation $i\rightarrow i'$ of $1,\ldots , r$ such that $x_i$ is similar to $y_{i'}$ (generate similar right ideals, and, therefore, left ideals \cite{Cohn3}).

\begin{lemma}(\cite{Cohn3}, Theorem 7.11.8) \label{le:co}  A group algebra $K(F)$ of a free group is a UDF. Any two principal right (left) ideals in $K(F)$ with nonzero intersection have a sum and intersection which are again principal.\end{lemma}

An element $c\in K(F)$ is called {\em rigid} if equality $c=ab'=ba'$ implies that either $a$ is a left factor of $b$ or $b$ is a left factor of $a$.
A relation $ab_1=ba_1$ in $K(F)$ is called comaximal if there exist $c,d,c_1,d_1\in K(F)$ such that $da_1-cb_1=ad_1-bc_1=1$. 
\begin{lemma} \label{le:comax} (\cite{Cohn3}, Proposition 3.2.9) Given any two complete atomic factorizations of an element in $K(F)$, one can pass from one to the other
by a series of comaximal transpositions. If a comaximal transposition $ab_1=ba_1$  exists, then $aK(F)+b_1K(F)=K(F).$
\end{lemma}

 \begin{lemma}(\cite{Cohn3}, Proposition 3.2.9 and Proposition 3.3.6)\label{un}  Let $c=a_1\ldots a_n$ be a decomposition  of an element $c$ into atoms, $c\in R=K(F)$. If for any $i=1,\ldots n-1$,  $a_i$ and $a_{i+1}$ generate a proper ideal, then $c$ is rigid.
\end{lemma}

We recall Bergman's ordering of the free group.  Let $\mathbb Q\langle\langle Y\rangle\rangle$ be the free power series ring over the rationals. 
If $|Y| = r$ , the elements $x=1+y(y\in Y)$with inverses $\Sigma (-1)^ny^n$ form the free generators of a free group $F(X)$ of rank $r$. 
One then orders the units of $\mathbb Q\langle\langle Y\rangle\rangle$ with constant term 1 by ordering lexicographically for each $n$ the
$\mathbb Q$-vector space of polynomials homogeneous of degree $n$, and comparing two elements by looking at the
first degree in which they disagree, and seeing which has ÔÔlargerÕÕ component in that degree.  This induces the total order on the free group .
Below  we order  $y_1< y_2\ldots <y_n$  polynomials of degree 1, and  $y_2y_1>y_1y_2>y_1^2$.  Then, for example, $1+2y_1+3y_2<1+3y_1+2y_2$. Then $x_1>x_2$ because $x_1=1+y_1, x_2=1+y_2$, so they disagree in degree 1 and $1+y_1$ corresponds to a tuple $(1,0)$ while $1+y_2$ corresponds to a tuple $(0,1)$ .

\begin{lemma} \label{ab}  Element $(1-x_1^mx_2^n)$ is irreducible in the algebra of Laurent polynomials $K[x_1,x_2,x_1^{-1},x_2^{-1}]$ if and only if $m$ and $n$ are relatively prime.
\end{lemma}
\begin{proof} If $m$ and $n$ are not relatively prime then $x_1^mx_2^n$ is a proper power and $(1-x_1^mx_2^n)$ is reducible.

Suppose now that $m$ and $n$ are relatively prime, wlog $m,n>0$.  The order on the free group $F$ also gives the order on the free abelian group $F/[F,F]$ because elements in $[F,F]$ are smaller than elements not in $[F,F]$. Then $x_1^kx_2^t>1$ for $k>0$  and $x_1^kx_2^t<1$ for $k<0$ .
Indeed, in the algebra of formal power series for $k,t\geq 0$ we have $x_1^kx_2^t=(1+y_1)^k(1+y_2)^t$, and the lowest  term of  degree 1 is $ky_1$.


Assume now that  $1-x_1^mx_2^n=f(x_1,x_2)g(x_1,x_2)$ is a non-trivial factorization. Then all the monomials in $f(x_1,x_2), g(x_1,x_2)$ are greater than or equal to $1$. Substitute $x_1=z^{-n}, x_2=z^m$, then $f(z^{-n},z^m)g(z^{-n},z^m)=0$. Therefore one of $f(z^{-n},z^m)$ or $g(z^{-n},z^m)$, 
wlog $f(z^{-n},z^m)$ is the zero polynomial. This means that  all the monomials cancel, i.e. if 
$$f(x_1,x_2)=\sum a_{ij} x_1^ix_2^j,$$ then
$\sum _{-ni+mj=k} a_{ij}=0.$ We also have $0\leq i\leq m-1$ because for any $i,j$,  $x_1^ix_2^j\geq 1$.  
We cannot have $mj-ni=mj_1-ni_1$ because this would imply $m(j-j_1)=n(i-i_1)$ and (since $m$ and $n$ are relatively prime),
$j_1=j \ (mod\ n)$ and  $i_1=i \ (mod\ m)$. Therefore $i=i_1$ and $j=j_1$. Therefore when we substitute $x_1=z^{-n}, x_2=z^m$ none of the monomials in 
 $f(x_1,x_2)$ will cancel out, and  $f(z^{-n}, z^m)$ cannot be a zero polynomial. This contradicts to the assumption about the non-trivial factorization $1-x_1^mx_2^n=f(x_1,x_2)g(x_1,x_2).$
\end{proof}
Notice that in the case  $h=g^k$, $k>1$,  $(1-h)$ is reducible. Indeed, $(1-h)=(1-g)(1+g+g^2+\ldots +g^{k-1})$. 
\begin{lemma} \label{irr}  Let $F=F(x_1,\ldots ,x_k)$ and $h=x_1^mx_2^nz$, where $m\neq 0$ or $n\neq 0$,  $z\in [F,F]$.
If $(1-h)$ is reducible in $K(F)$, then $x_1^mx_2^n$ is a proper power in the quotient  $F/[F,F]$. \end{lemma}

\begin{proof}  Suppose $x_1^mx_2^n$ is not a proper power in the quotient $F/[F,F]$. Replacing $x_i$ by $x_i^{-1}$ if necessary, we can assume $h>1$ and $m,n>0.$ Suppose $1-h=f(x_1,x_2)g(x_1,x_2)$, and $f,g$ are not units. We can assume that 1 is the smallest term of $f(x_1,x_2)$ and $g(x_1,x_2)$. By Lemma \ref{ab}  the image of one of $f(x_1,x_2)$ or $g(x_1,x_2)$, wlog $f(x_1,x_2)$ must be equal $1$ in the abelianization.  Then the oldest term of $g(x_1,x_2)$ must be $x_1^mx_2^nz_1,$ where $z_1\in [F,F].$ Therefore the oldest term of $f(x_1,x_2)$ must belong to $[F,F]$ because otherwise the oldest term of $1-h$ would be larger than $h$. Therefore all the monomials in $f(x_1,x_2)$ must be in $[F,F]$ because  monomials not in $[F,F]$ which are greater than 1  are larger than monomials in $[F,F]$. Indeed,  such monomials not in $[F,F]$  have a positive coefficient of $y_1$ or $y_2$ when embedded into formal power series and elements in $[F,F]$ have these coefficients zero.  Then $g(x_1,x_2)=\sum f_{ij}x_1^ix_2^j,$ where for each $i,j$, $f_{ij}\in K([F,F])$.  Then $f(x_1,x_2)f_{00}=1$, that  contradicts the assumption that $f(x_1,x_2)$ is not a unit.
 \end{proof}

\begin{lemma} \label{le:uni}  Let $a_m=x_1x_2^2x_1^2x_2^2x_1^3\ldots x_1^mx_2^m$, $m>1$ and $(1-a),(1-b),(1-c)$ are irreducible in $K(G)$.   Let $C_i=(1-c)(1-a)^{1+i} (1-b).$ 
Then   $w=C_1(1-a_2)C_2(1-a_3)A_3\ldots C_{m-1}(1-a_m)C_m$
 is a rigid element.
\end{lemma}

\begin{proof} The algebra $K(F)$ is embedded into the ring of formal power series $K\langle\langle Y\rangle\rangle$ (\cite{Cohn3}, Ex. 7.11.9). An embedding is given by mapping each $x_i$ into $1+y_i$. Notice that  the images of $(1-a)$, $(1-a_m)$,  $(1-b), (1-c)$ in $K\langle\langle Y\rangle\rangle$ do not contain a constant term, and therefore, are non units in the ring of formal power series. Also  by Lemma \ref{irr} elements $(1-a_m)$ are atoms in $K(F)$,  Non-units in the ring of formal power series (which is a rigid unique factorization domain) generate the maximal  proper ideal. That ideal does not contain the identity element. Therefore elements $(1-a), (1-b), (1-c), (1-a_m)$ generate a proper ideal in  $K(F)$. Therefore $w$ is a rigid element in $K(F)$ by Lemma \ref{un}.
\end{proof}

\begin{lemma} \label{10}  Let $a,b,c\in F$ such that  $a,b,c >1$, $ab\neq ba, ac\neq ca$ and $(1-a),(1-b),(1-c)$ are irreducible in $K(G)$. If for some numbers $m,n$ and invertible elements $g,h$
there is an equality $g(1-c)^n(1-a)^m(1-b)^n=(1-c)^n(1-a)^m(1-b)^nh$, then $g=h\in K$.
\end{lemma}
\begin{proof} Since $g$ is invertible $g\in KF$. Suppose $g\not\in K,$  WLOG we can assume that $g\in F$. Since $a,b,c >1$, the smallest monomial in the left side of the equality is $g$ and in the right side is $h$, therefore $g=h$. Then $(1-c)^n(1-a)^m(1-b)^n$ must be in the centralizer of $g$ in $K(F)$. By Lemma \ref{pr:centralizer-hyp}, the centralizer of an element in $F$ in $K(F)$ is isomorphic to the group algebra of a cyclic group. But $(1-c)^n(1-a)^m(1-b)^n$ does not belong to a group algebra of a cyclic group because $a,b,c$ do not commute. Therefore  $(1-c)^n(1-a)^m(1-b)^n$ does not commute with $g$, contradicting the assumption $g\not\in K$. \end{proof}

\begin{lemma} \label{le:uni1}  Let $a_m=x_1x_2^2x_1^2x_2^2x_1^3\ldots x_1^mx_2^m$, $b_m=x_1^mx_2^{m+1}, c_m=x_1^{m+1}x_2^{m+2}$.
  Let $A_{m,i}=(1-c_m)(1-a_m)^{1+i} (1-b_m).$ 
Let $x_{i1}^{e_1}\ldots x_{im}^{e_m}$, where $e_1,\ldots ,e_m\in \{\pm 1\}$ be a monomial in $F$. Then   $w=A_{m,0}(1-x_{i1}^{e_1})A_{m,1}(1-x_{i1}^{e_1}x_{i2}^{e_2})A_{m,2}\ldots A_{m,m-1}(1-x_{i1}^{e_1}\ldots x_{im}^{e_m})A_{m,m}$
uniquely defines the monomial $x_{i1}^{e_1}\ldots x_{im}^{e_m}.$ 

Moreover, if $p_1,\ldots, p_k\in K(F)$  such that all the monomials have length less than $m/2$, then for any $q_1,\ldots ,q_k\in K(F)$
$$A_{m,0}p_1A_{m,1}p_2A_{m,2}\ldots A_{m,k-1}p_kA_{m,k}=A_{m,0}q_1A_{m,1}q_2A_{m,2}\ldots A_{m,k-1}q_kA_{m,k}$$ implies that for any $i,$
$p_i=\delta _i q_i$, where $\delta _i\in K$ and $\delta _1\ldots \delta _k=1$.
\end{lemma}
\begin{proof}   We will prove the first statement. Notice that by Lemma \ref{irr}, $(1-a_m), (1-b_m), (1-c_m)$ are irreducible in $K(F)$ and $a_m, b_m, c_m$ are larger that all the monomials 
$x_{i1}^{e_1}\ldots x_{ik}^{e_k}, k\leq m.$ 

Assume that $w=A_0(1-x_{i1}^{e_1})A_1(1-x_{i1}^{e_1}x_{i2}^{e_2})A_2\ldots A_{m-1}(1-x_{i1}^{e_1}\ldots x_{im}^{e_m})A_m$ does not
uniquely define all the monomials $x_{i1}^{e_1}\ldots x_{ik}^{e_k}.$ This means that there is another decomposition of $w$ with different monomials. Then one can pass from one to the other by a series of comaximal transpositions.  An atom in the decomposition 
of $(1-x_{i1}^{e_1}\ldots x_{im}^{e_m})$ after multiplying by a suitable group element on the left will have the form $(1+\alpha_1g_1+\ldots \alpha_sg_s),$  where all $g_i\in F$, $1<g_1<\ldots <g_s\leq x_{i1}^{e_1}\ldots x_{im}^{e_m}$.
One should have a comaximal  atomic transposition of  such an atom $(1+\alpha_1g_1+\ldots \alpha_sg_s)$, with the atom $(1-b_m)=(1-x_1^mx_2^{m+1}):$

\begin{equation}\label{co}
(1+\alpha_1g_1+\ldots \alpha_sg_s)(1-x_1^mx_2^{m+1})=(1+\gamma _1h_1+\ldots +\gamma _kh_k)(1+\beta_1f_1+\ldots \beta_tf_t),\end{equation}
where $h_j,f_j\in F$ for all $j$.

If $(1+\alpha_1g_1+\ldots \alpha_sg_s)$ is not a unit in the ring of formal power series, then $(1+\alpha_1g_1+\ldots \alpha_sg_s)$ and $(1-x_1^mx_2^{m+1})$  generate a proper ideal in the ring of formal power series and, therefore, in $K(F)$ and such a  comaximal transposition is impossible by Lemma \ref{le:comax}. Therefore $(1+\alpha_1g_1+\ldots \alpha_sg_s)$ is a unit in the ring of formal power series (that is $1+\alpha _1+\ldots +\alpha _s\neq 0$).
We now substitute $x_1=z^{-m-1}, x_2=z^m$ into this relation, then the left side becomes zero. This substitution sends one of the elements $1+\gamma _1h_1+\ldots +\gamma _kh_k$ or $1+\beta_1f_1+\ldots \beta_tf_t$ to zero. Suppose  $1+\gamma _1h_1+\ldots +\gamma _kh_k$ becomes zero. Then for each monomial $x_1^ix_2^j$ in the abelianization of $1+\gamma _1h_1+\ldots +\gamma _kh_k$ there should be a monomial $x_1^ix_2^jx_1^mx_2^{m+1}.$ There are only two possible  cases:

a) $1+\gamma _1h_1+\ldots +\gamma _kh_k=u_1+u_2$, where the largest monomial  in $u_1$ is less than $x_1^mx_2^{m+1}$ and   the smallest monomial  in $u_2$ is greater or equal than $x_1^mx_2^{m+1}$. Then $x_1^mx_2^{m+1}(1+\alpha_1g_1+\ldots \alpha_sg_s)=-u_2(1+\beta_1f_1+\ldots \beta_tf_t)$ in $K(F)$, and since $1+\alpha_1g_1+\ldots \alpha_sg_s$ is an atom, $u_2$ must be a monomial, $u_2= -x_1^mx_2^{m+1}$. Similarly $u_1$ must be a  monomial, $u_1=1$. 
Then $1+\alpha_1g_1+\ldots \alpha_sg_s=1+\beta_1f_1+\ldots \beta_tf_t,$  and this element must commute with $1-x_1^mx_2^{m+1}$. This is impossible 
because the centralizer of $x_1^mx_2^{m+1}$ is the group algebra of the cyclic group generated by $x_1^mx_2^{m+1}$.

b) All the monomials in $1+\gamma _1h_1+\ldots +\gamma _kh_k$ are in $[F,F]$. Since $1+\gamma _1h_1+\ldots +\gamma _kh_k$ becomes zero when we substitute 
$x_1=z^{-m-1}, x_2=z^{m}$, we have $1+\gamma _1+\ldots +\gamma _k=0$. Therefore  $1+\gamma _1h_1+\ldots +\gamma _kh_k$ also becomes zero when we substitute $x_1=z^{-1}, x_2=z$, .   Since $1+\alpha_1+\ldots +\alpha_s\neq 0$,  $(1+\alpha_1g_1+\ldots \alpha_sg_s)$  does not become zero after this substitution. Therefore the right side of equation (\ref{co}) becomes zero and the left side does not,  this is a contradiction.

The proof of the second statement is similar, we show that we  cannot make a comaximal transposition (\ref{co}) in the decomposition $$A_{m,0}p_1A_{m,1}p_2A_{m,2}\ldots A_{m,k-1}p_kA_{m,k}.$$ We just have to notice that when we multiply any $p_i$ by a monomial $g$ to obtain an element in the form 
$(1+\alpha_1g_1+\ldots \alpha_sg_s)$, where all $g_i\in F$, $1<g_1<\ldots <g_s$, then $g_s$ is a product of the maximal monomial in $p_i$ and $g^{-1}$,  where $g$ is the minimal monomial in $p_i$ that is less than or equal to 1. The length of $g^{-1}$  the maximal monomial in $p_i$ is less than $m/2$, therefore $g_2$ has length less than $m$.  Since $x_1<\ldots < x_n$,  $g_s< a_m,b_m,c_m.$
\end{proof}

\begin{lemma} \label{le:limit} Let $G$ be a limit group, $F\leq G$, $A_0,\ldots ,A_m$ the same  as in Lemma \ref{le:uni}. 
Then the equality  $w=A_{m,0}(1-v_1)A_{m,1}(1-v_2)A_{m,2}\ldots (1-v_m)A_{m,m}$, where the length of each $v_i\in G$ is less or equal than $m$,  uniquely defines elements $v_1,\ldots,v_m.$

\end{lemma}
\begin{proof}   The equality $ A_{m,0}(1-v_1)A_{m,1}\ldots (1-v_m)A_{m,m}-A_{m,0}(1-w_1)A_{m,1}\ldots (1-w_m)A_{m,m}=0$ in $K(F)$  (in $K(G)$) is equivalent to the fact that some monomials are the same and the sum of the coefficients of the same monomials is zero. Since we do not know $w_1,\ldots ,w_m$ there are different possible systems of equalities but there is a finite number of possible systems. Therefore this equality in $K(F)$ (resp. in $K(G)$) is equivalent to the disjunction of  systems of equalities in $F$  (resp. in $G$). Denote this disjunction 
by $\psi (v_1,w_1,\ldots ,v_m,w_m).$    The formula $$\forall v_1,\ldots ,v_m,w_1,\ldots , w_m (\psi (v_1,w_1,\ldots ,v_m,w_m)\implies v_i=w_i)$$ in the group language is true in $F$.
But  the universal theories of $F$ and $G$ (with constants in $F$) are the same, see, for example, \cite{ICM}. Therefore this formula is also true in $G$.
\end{proof}
\begin{lemma} \label{le:a-m}
The set of pairs  $B = \{(a_m,m) \mid  m \in \N_P\}$ is definable in $K(F)$ uniformly in $K$ and $x_1, x_2 \in X$ such that $x_1 \neq x_2$.
\end{lemma}
\begin{proof}
By Lemmas \ref{le:uni} and \ref{10}, monomial  $ a_m$ is  completely determined by the number $m$ and the  following conditions that define  the element $w=C_1(1-a_2)C_2(1-a_3)A_3\ldots C_{m-1}(1-a_m)C_m$ :
\begin{itemize}
\item [a)]  $w = C_1(1-a_2)C_{2}v$ for some $v \in K(F)$;
\item [b)] $w = w_1C_{m-1}(1-a_m)C_{m}$ and $w_1 \in K(F)$;
\item [c)] (recursion) for any $i \in \N_P, 0 <  i < m$,  and any $ w_1,w_2,w_3 \in K(F)$,   if $w=w_1C_{i-1}w_2C_{i}w_3$ then $w_3 = (1-\beta ^{-1}(\beta -w_2)x_1^ix_2^i)C_{i+1}v_1$ for some $v_1 \in K(F)$ and unique $\beta \in K,$ such that $(\beta -w_2)$ is invertible. 
\item [d)] (uniqueness) if $w_t = w_1C_{j}w_2 = w_1^\prime C_{j}w_2^\prime$ for some $w_1,w_1^\prime \in K(F)$, $w_2,w_2^\prime \in K(F)$ and  $j <m$ then $w_1 = \alpha w_1^\prime, 
 w_2 = (\alpha)^{-1}w_2^\prime$, $\alpha \in K$.
  \end{itemize}
One can write a formula in $K(F)$ with variables $m\in \N_P\subseteq K(F)$ and $z\in K(F)$ that states that there exists a word $w$ satisfying these four conditions. The formula  is true  for all the pairs $(m,a_m)$ and only for them.
\end{proof}

\section{Bi-interpretability of   $S(K,\N)$ and $K(F)$ and applications}\label{biint}
\subsection{Bi-interpretability}
\begin{definition}\label{bi-in}
Algebraic structures  $\MA$ and $\MB$ are called {\em bi-interpretable} if  the following conditions hold:
\begin{itemize}
\item $\MB$ is interpretable in $\MA$ as $\MB^{\diamond}$, $\MA$ is interpretable in $\MB$ as $\MA^{\diamond},$ which by transitivity  implies that $\MA$ is interpretable in $\MA$, say by $\MA^{\diamond\diamond}$, as well as $\MB$ in $\MB$, say as $\MB^{\diamond\diamond}$.
\item  There is an  isomorphism $\MA \to \MA^{\diamond\diamond}$ which is definable in $\MA$ and there is an isomorphism $\MB \to \MB^{\diamond\diamond}$ definable in $\MB$.
\end{itemize}
\end{definition}
In this section we will prove the following result.
\begin{theorem} \label{th:bi} Let $K$ be an infinite field. The structures $S(K,\N)$ and $K(F)$ are bi-interpretable.\end{theorem}
By Lemma \ref{le:list-superstructure} the list superstructure $S(\Z,\N)$ is 0-interpretable in $\N$.
 Hence $S(\Z,\N)$ is interpretable in  $\N_P$ for any non-invertible polynomial $P \in 
K(F)$, that is the sum of at least three monomials, uniformly in $K,X$ and $P$. This allows us to assume that  the tuples from $S(\Z)$ and operations and predicates from $S(\Z,\N)$ are definable (or interpretable)  in  $\N_P$.

 Consider the following interpretation of the free group $F=F(X)$  in $S(\Z,\N)$.  A reduced monomial $M=x_{i_1}^{e_1}\ldots x_{i_m}^{e_m}\in F$, where $e_j\in\{\pm1\}$, can be uniquely  represented by the tuple of integers $ t=(e_1i_1,\ldots, e_mi_m)$ with the property $e_ji_j\neq -e_{j+1}i_{j+1}.$   Denote by $T$ the set of all such tuples $t= (e_1t_1, \ldots,e_mt_m) \in S(\Z) \leq S(\Z,\N)$. Then   for any $i$ one has $1\leq t_i \leq n.$  The multiplication in $F$  corresponds to the concatenation of the tuples  in  $T$, and removing the maximal cancelable pieces. If tuple $s$ corresponds to $M=x_{i_1}^{e_1}\ldots x_{i_m}^{e_m}$,  then the tuple $s_1$ corresponding to $M^{-1}$ is defined by the conditions $\ell (s)=\ell (s_1)$ and for any $i\leq \ell (s)$, $t(s,i,a)\iff t(s_1, \ell(s)-i, -a)$. Hence the construction above gives an interpretation of $F$  in $S(\Z,\N)$.  Denote the interpretation of $S(\Z,\N)$ in $K(F)$  through $\N_P$  by $S(\Z,\N)_P$,  denote the interpretation of $S(\Z)$ by $T_P$ and the corresponding interpretation of $F$ by $\MM_{X,P}$  (it is uniform in $K, |X|$ and $P$).  A tuple of integers $t\in S(\Z)$
 is represented in $\MM_{X,P}$ as $t^{\diamond}\in T_P$.

\begin{lemma} \label{le:M-X-c}
 The canonical isomorphism $F(X) \to \MM_{X,P}$ defined by the map $M=x_{i_1}^{e_1}\ldots x_{i_m}^{e_m} \to t^{\diamond}\in T_P$ is definable in $K(F)$ uniformly in $K$, $X$, and $P$.
\end{lemma}
\begin{proof}
The case $|X| = 1$ follows from Lemma \ref{th:HF-group2}. Assume now that $|X| \geq 2$. 
Below we construct a formula $\Phi(x,y,P)$ of the language $L_X$  such that 
$K(F) \models \Phi(M,t^{\diamond},P)$ if and only if the monomial $M \in F$ corresponds to the tuple $t^{\diamond} \in T_P$.  Here $t^{\diamond}$ is the encoding of the tuple $t\in S(\Z,\N)$ corresponding to the monomial $M$.

Take an arbitrary $t^{\diamond} \in T_P$. Notice that the length function $\ell: T_P \to \N$ that gives the length of the tuple $t^{\diamond}$ is definable in $S(\Z,\N)_P$, as well as in $K(F)$ (with the constant $P$). 
  Hence there is a formula $\phi_1(x,y,P)$ in $L_X$ such that in the notation above $K(F) \models \phi_1(t^{\diamond},m^{\diamond},P)$ if and only if $m = \ell(t)$. Similarly, there exists a formula $\phi_2(x,y,z,P)$ in the language $L_X$ such that $K(F) \models \phi_2(t^{\diamond},i^{\diamond},s^{\diamond},P)$ if and only if $t = (e_1t_1, \ldots,e_mt_m) \in T, i \in \mathbb{N}$, $1 \leq i \leq \ell(t)$, and $s = e_it_i$.

Our goal is to write a formula that is only true for corresponding pairs $(M_t, t^{\diamond}).$  Therefore we have to find a way, given $t^{\diamond}$ (and, therefore, given $t$) to recursively construct corresponding $M_t\in F(X)\subset K(F)$  so that the definition can be presented in the language $L_X$ and does not depend on $\ell(t)$ (contains $\ell(t)$  only as a variable). This is what we are going to do now.

 Let $A_{m,i}$ be defined as in Lemma \ref{le:uni1}.  Now for a tuple $t = (e_1t_1, \ldots,e_mt_m)\in T$ (or, equivalently, $t^{\diamond}$ in $T_P$) we define recursively an element $w_t$ that is as follows 
\begin{multline}\label{eq:w}
w_t=A_{m,0}(1-x_{t_1}^{e_1})A_{m,1}(1-x_{t_1}^{e_1}x_{t_2}^{e_2})\ldots  A_{m,m-1}(1-x_{t_1}^{e_1}\ldots x_{t_m}^{e_m})A_{m,m}.\end{multline}

By Lemma \ref{le:uni1}, the element  $w_t$ is completely determined  by the  tuple $t$ and the  following conditions:
\begin{itemize}
\item [a)] (head) $w_t = A_{m,0}(1-x_{t_1}^{e_1})A_{m,p+1}v$ for some $v \in K(F)$;
\item [b)] (tail) $w_t = w_1A_{m,m-1}(1-x_{t_1}^{e_1}\ldots x_{t_m}^{e_m})A_{m,m}$ where $m = \ell(t)$ and $w_1 \in K(F)$;
\item [c)] (recursion) for any $i \in \N_P, 0 <  i < m$,  and any $ w_1,w_2,w_3 \in K(F)$,   if $w=w_1A_{m,i-1}w_2A_{m,i}w_3$ then $w_3 = (1-\beta ^{-1}(\beta -w_2)x_{t_i}^{e_i})A_{m,i+1}v_1$ for some $v_1 \in K(F)$ and unique $\beta \in K,$ such that $(\beta -w_2)$ is invertible. 
\item [d)] (uniqueness) if $w_t = w_1A_{m,j}w_2 = w_1^\prime A_{m,j}w_2^\prime$ for some $w_1,w_1^\prime \in K(F)$, $w_2,w_2^\prime \in K(F)$ and  $j <m$ then $w_1 = \alpha w_1^\prime, 
 w_2 = (\alpha)^{-1}w_2^\prime$, $\alpha \in K$.
  \end{itemize}
 One can write down the condition a) -d) by formulas of the language $L_X$.  Notice, that by construction $w_t =w_1A_{m,m-1}(1-M_t)A_{m,m}$ for some $w_1$  as above.
Therefore we can define the set $\{w=\beta (1-M_t), \beta\in K\}$.  Then $u=M_t$ is defined  by the formula 
$$\phi _3 (u,w)=\exists \beta\in K (\beta\neq 0)\wedge (w-\beta) {\rm is\  invertible} \wedge u=\beta ^{-1} (\beta -w).$$

 Hence there  exists a formula $\phi_4(y_1,y_2,y_3,y_4,y_5,y_6)$ of the  language $L_X$ such that 
$$
K(F) \models \phi_4(t,w,u,a,m,P) \Longleftrightarrow  t \in T_P, w= w_t, u = M_t, a = a_m, m\in \N_c.
$$  Notice that we use the fact that  the set of pairs $(m,a_m)$ is definable by Lemma \ref{le:a-m}.
Therefore, the formula 
$$\phi_5(t,u,P, X) = \exists w \exists a \exists m \phi_4(t,w,u,a,m,P)$$
 defines all the pairs $(t,M_t)$ for $t \in T_P$, $M_t \in F$. This formula defines an isomorphism of interpretations $F(X) \to \MM_{X,P}$ given  by the map $M \to t_M$, as required. In this formula $X$ is the set of parameters, showing that the formula is in  the language $L_X$.

\end{proof}

We interpret $K(F)$ in $S({K}, \N)$ as follows. Let $\MM_{X,P}$ be the interpretation of $F(X)$ in $S(\Z,\N)$ as above. For an element  $f=\sum_{i= 1}^s \alpha _iM_i \in K(F)$, where $\alpha_i \in K, M_i \in F(X)$,  we associate a pair  $q_f = (\overline\alpha, \overline t)$, where $\overline\alpha = (\alpha_1, \ldots, \alpha_s)$, $\overline t = (t_{M_1}, \ldots,t_{M_s})$.  This gives interpretation, say $K(F)^{\diamond}$, of $K(F)$ in $S(S(K,\N), \N)$  (where we have sequences of elements from the union of $\Z$ and $K$), and, therefore, in $S(K,\N)$ by Corollary \ref{c2}. Indeed, one can interpret addition and multiplication in $K(F)$ similarly to how it was done in \cite{assoc}, Theorem 2 for a free associative algebra. By Proposition \ref{th:S(F,N)-non-comm},  $S(K,\N)$ is 0-interpretable  in $K(F) $  uniformly in $K$ and $F$, hence by transitivity of interpretations, we have interpretation  $K(F)^{\diamond\diamond}$ in $K(F)$.

\begin{lemma} \label{le:bi-int-iso}
The isomorphism between $K(F)^{\diamond\diamond}$ and $K(F)$ is definable.
\end{lemma}
\begin{proof} We will first consider the case  $|X|=1$. In this case $K(F)=K[t,t^{-1}]$. An element $f(t)=t^{-m}(\alpha _0+\alpha _1t+\ldots ,\alpha _nt^n)$ in $K[t,t^{-1}]$ is interpreted as a pair $((\alpha _0,\alpha _1,\ldots ,\alpha _n), m)\in S(K,\mathbb N),$ where $m\in\mathbb N$. Such a pair is interpreted back in $K[t,t^{-1}]$ as the equivalence class of a pair $s_P=(\sum \alpha _iP^i, P^m)$, where $P$ is the sum of al least three monomials. The relation $((f(P), \{t^kf(t), k\in\mathbb N\})$ is definable as in the second part of Lemma 2. The relation $((f(P),P^m-1), t^{-m}f(t))$ is then also definable. Therefore the isomorphism $K[t,t^{-1}]^{\diamond\diamond}\rightarrow K[t,t^{-1}]$ is definable.

Let now $|X|\geq 2$. Given a pair $q = (\overline\alpha, \overline t)$, where $\overline \alpha =  (\alpha_1, \ldots, \alpha_s)$, $\overline t = (t_1, \ldots,t_s)$ define a polynomial $f_q$ as follows. 
Let $m= 2\max\{\ell(t_i) \mid i = 1, \ldots,s\}$ and $a_m,b_m,c_m \in F(X)$ defined above. Notice that having a sequence $\overline \alpha$ we can define a sequence $\overline \gamma =(\gamma _1,\ldots ,\gamma _s)$, where $\gamma _i=\alpha _1+\ldots +\alpha _i$ by $\gamma _i=\gamma _{i-1}+\alpha _i.$
We define
\begin{equation} \label{eq:p-unique}
f_q=A_{m,1}h_1A_{m,2}h_2\ldots A_{m,s}h_{s}A_{m,s+1},
\end{equation}
 where $h_1=\alpha _1-\alpha _1M_1, h_{i+1}=h_i+\alpha _{i+1}-\alpha _{i+1}M_{i+1}$, so in particular, $h_s = \alpha_1+\ldots \alpha _s-f$.
 Assume that  $s\geq 2.$
The polynomial $f_q$ is uniquely determined by the following conditions:
\begin{itemize}
\item [i)] $f_q = A_{m,1}h_1A_{m,2}h_2A_{m,3}g$,  $g\in K(F)$;
\item [ii)] for any $1\leq i \leq s$ if  $f_q = g_1A_{m,i}g_2 = g_1^\prime A_{m,i}g_2^\prime$,  then $g_1 = g_1^\prime, g_2 = g_2^\prime$ (up to a multiplicative constant from $K$).  
\item [iii)] If $s>2$, for any $1\leq i \leq s-2$, if $f_q = g_0A_{m,i}g_1A_{m,i+1}g_2A_{m,i+2}g_3$, and $g_2-g_1=\alpha _{i+1}+x$, where $x$ is invertible and  $x\not\in K$,  then $g_3 = h_{i+2}A_{m,i+3}g_4$, where $h_{i+2} =  g_2 +\alpha_{i+2} -\alpha_{i+2}M_{i+2}$.
 \item [iv)] $f_q = g_5A_{m,s+1}$.
\end{itemize}

Indeed, to show that i)-iv) determine $f_q$ completely one needs the uniqueness of the decomposition (\ref{eq:p-unique}) up to  scalar multiples, which follows from Lemma \ref{le:uni1} and the property that $h_{i+1}-h_i=\alpha _{i+1}+x,$ where $x$ is invertible and not in $K$. This isomorphism does not depend on $P$ and is defined uniformly in $P$.  Knowing $h_s$ we can find
$\alpha _1M_1+\ldots +\alpha _sM_s=- h_s+\alpha _1+\ldots +\alpha_s$. \end{proof}

\begin{lemma} The isomorphism between $S(K,\N)$ and $S(K,\N)^{\diamond\diamond}$ is definable. 
\end{lemma}
\begin{proof}  Recall, that in Proposition \ref{th:S(F,N)-non-comm} for a given non-invertible polynomial $P\in K(F)$ which is the sum of at least three monomials one  interpret $S(K,\N)$ in $K(F)$ by $S(K,\N)_P$, using the parameter $P$ uniformly in $K, X,$ and $P$ by 
representing a tuple $s=(\alpha _1,\ldots ,\alpha_m)\in S(K,N)$ as  a pair $s_P = (\sum_{i = 0}^m \alpha_iP^i, P^m) \in S(K)_P$.  For any other such a polynomial $Q$  there is another interpretation $r_Q = (\sum_{i = 0}^m \alpha_iQ^i, Q^m) \in S(K)_Q$. And there is a formula which states that for any $0 \leq i \leq l(s) = l(r)$ the sequences   $s_P$ and $r_Q$ have the same $i$ terms in $K$. Hence   
the set of pairs 
$$
\{ (\sum_{i = 0}^m \alpha_iP^i, \sum_{i = 0}^m \alpha_iQ^i) \mid \alpha_i \in K, m \in \N\}
$$
 is also definable in $K(F)$ uniformly in $K, X, P$ and $Q$. This gives an isomorphism $ S(K,\N)_P \to S(K,\N)_Q$ definable in $K(G)$ uniformly in $K$, $G$, $P$, and $Q$ and factorizing by this equivalence relation on the set of pairs we obtain $0$-interpretation of $S(K,\N)$ in $K(F)$ denoted by $S(K,\N)^{\diamond}.$ 
 
 The interpretation $S(K,\N)^{\diamond\diamond}$ in $S(K,\N)$ is obtained by 
 defining in $S(S(K,\mathbb N),\mathbb  N)$  the images $(\bar\beta ,\bar t)$ of the set of pairs $s_P = (\sum_{i = 0}^m \alpha_iP^i, P^m)$ for different $P$ and factorizing this image by the image of the above equivalence  relation on these pairs. Let $(\bar\gamma ,\bar t_1)$ be the image of $P-1$.  The relation $((f(P),P-1), x_1^{-1}f(x_1))$ is definable in $K(F)$ as in the previous lemma, therefore the relation $(\phi((\bar\beta ,\bar t),(\bar\gamma ,\bar t_1)), (\alpha _1,\ldots ,\alpha_m))$, where $\phi$ is the interpretation of $S(S(K,\N),\bar N)$ in $S(K,\mathbb N)$, is definable in $S(K,\mathbb N)$. This relation defines the isomorphism between  $S(K,\mathbb N)^{\diamond\diamond}$  and $S(K,\mathbb N)$.  
 
\end{proof}
Theorem \ref{th:bi} is proved.

\begin{theorem} \label{th:bi1} Let $L$ be a  non-abelian limit group and $K$ be an infinite field, then the structures $S(K,\N)$ and $K(L)$ are bi-interpretable.
Moreover, $L$ is interpretable in $S(\Z,\N)$ as $L^{\diamond}$ and the canonical isomorphism $L\rightarrow L^{\diamond\diamond}$ in $K(L)$ is definable. \end{theorem}
\begin{proof} Since $L$ is non-abelian  it contains a free subgroup $F$. We fix a finite set of generators of $L$ including generators of $F$. Since the word problem in $L$ is solvable there exists a recursive  enumeration of geodesics in the Cayley graph of $L$, one (left lexicographically minimal) geodesic for each element of $L$. Lemma \ref{le:limit} implies that one can use the same construction as in $K(F)$ to prove bi-interpretation. 

\end{proof}

 \subsection{Definability of bases of $F$ in $K(F)$}
 
 We continue to use notation from the previous sections. In particular, below   $K$ is  an infinite field,  $X = \{x_1, \ldots,x_n\}$ is a set with $n = |X| \geq 2$, $F=F(X)$.

 In this section we prove the following result.  
 \begin{theorem} \label{th:bases} 
 The set of all free bases of $F$ is 0-definable in $K(F)$.
 \end{theorem}
 \begin{proof}  By Lemma \ref{le:M-X-c}, $F$ is definable in $K(F)$. There is a definable  isomorphism between  $K(F)$ and $K(F)^{\diamond\diamond}$.  The set $Y=\{y_1,\ldots ,y_n\}$ forms a free base of $F$ if and only if it generates $F$ and for any different tuples $t_1$ and $t_2$ from the proof of Lemma \ref{le:M-X-c}, the monomials $M_{t_1}$ and $M_{t_2}$ are different.   So we write the formula 
$$ \forall u\in F\exists t \phi _5(t,u,P,Y)$$ meaning  that $Y$ generates $F$ and the formula
$$\forall  u_1,u_2\in F ((\phi _5(t_1,u_1,P,Y) \wedge \phi _5(t_2,u_2,P,Y))\implies (u_1=u_2\implies t_1=t_2))$$
  that says that for any two elements in $F$ and corresponding tuples $t_1$ and $t_2$,  if the tuples are different then the monomials $u_1$ and $u_2$ constructed using generating set $Y$ from these tuples are different.
 
 \end{proof}

 \subsection{Definability of the metric}
 
 Let $G$ be a group with a finite generating $X$. The word metric $d_X:G \times G \to \N$ on $G$ with respect to $X$ is defined for a pair $(g,h) \in G\times G$ as the length of a shortest word $w$  in the generators from $X \cup X^{-1}$ such that 
 $gw = h$.  This word $w$ is a geodesic between $g$ and $h$. This is precisely the geodesic metric on the Cayley graph of $G$ with respect to $X$, $Cay (G,X)$ that has elements of $G$ as vertices and directed edges labelled by the 
elements from the generating set $X$, so that for any $g\in G$ and $x\in X$ there is an edge $(g,gx).$ Below we view the metric space $G$ with metric $d_X$ as a structure $Met(G,X) = \langle G, \N, d_X\rangle$, where $G$ comes equipped with multiplication and the identity $1$, $\N$ is the standard arithmetic, and $d_X$ is the metric function $d_X: G \times G \to \N$.

 \begin{theorem} \label{th:metric}
Let $G$ be a limit group with  a finite generating set $X$, $K$ an infinite field. Then the metric space $Met(G,X)$ is definable in $K(G)$.
  \end{theorem}
\begin{proof}
For a non-abelian limit group the statement follows from Theorem \ref{th:bi1}. Indeed, every geodesic corresponds to a tuple in $S(K,\N)$, and the length of the tuple is the length of the geodesic.

For an abelian group it follows from Theorem \ref{th:HF-group}  and  Lemma \ref{th:HF-group2}. 
\end{proof}

\begin{definition}
A geodesic metric space is called $\delta -${\em hyperbolic} if for every geodesic
triangle, each edge is contained in the $\delta $ neighborhood of the union
of the other two edges. A group is hyperbolic if it is $\delta$-hyperbolic for some $\delta >0.$
\end{definition}

\begin{definition} Let $G$ be a hyperbolic group, with Cayley graph $Cay (G,X)$.  A subgroup $H$ is $k$-{\em quasiconvex} if for every $h\in H$, every geodesic from the identity to $h$ is in the $k$-neighborhood of $H$. 
\end{definition}

A subgroup $H\leq G$ is  {\em malnormal} if $gHg^{-1}\cap H$ is finite (in torsion free group is trivial) for any $g\in G-H.$

Theorem \ref{th:metric} implies the following results.

\begin{theorem} \label{th:geom} 
Let $G$ be a  limit group with  a finite generating set $X$, $K$ an infinite field. Then the set of geodesics in $G$ with respect to $X$ is definable in $K(G).$ 
\end{theorem} 
\begin{theorem}\label{th:geom1} Let $G$ be a non-abelian limit group with  a finite generating set $X$, $K$ an infinite field. 

1. Given a number $\delta$, there exists a formula $Con_{\delta }(y_1,y_n)$ such that for any elements $h_1,\ldots  h_n$ 
$K(G)\models Con_{\delta }(h_1,\ldots ,h_n)$  if and only if the subgroup generated by $h_1,\ldots ,h_n$ in $G$ is $\delta$-quasi-convex.

2. Given a number $\delta$, there exists a formula $Hyp_{\delta }(y_1,\ldots ,y_n)$ such that for any elements $h_1,\ldots  h_n$ 
$K(G)\models Hyp_{\delta }(h_1,\ldots ,h_n)$  if and only if the subgroup generated by $h_1,\ldots ,h_n$ in $G$ is $\delta$-hyperbolic.

3. For any word $w(y_1,\ldots ,y_m)$ there exists a formula $\phi_w(y, X)$ that defines in $K(G)$ the verbal subgroup of $G$ coresponding to 
$w(y_1,\ldots ,y_m)$. 
A free finite rank group in a variety defined by $w(y_1,\ldots ,y_m)$ is interpretable in $K(G)$.

4. There exists a formula $Mal (y_1, \ldots y_n)$ such that for any elements $h_1,\ldots  h_n$ 
$K(G)\models Mal (h_1,\ldots ,h_n)$  if and only if the subgroup generated by $h_1,\ldots ,h_n$ in $G$ is malnormal. 

\end{theorem}

\subsection{Definability of subrings and submonoids}

We use notation above and the following one. For elements $f_1, \ldots,f_n, h \in K(G)$ by $ring(f_1, \ldots,f_n)$ we denote the subring of $K(G)$ generated by $f_1, \ldots,f_n$, by $id(f_1, \ldots,f_n)$ we denote the ideal  of $K(G)$ generated by $f_1, \ldots,f_n$, by  $mon(f_1, \ldots,f_n)$ the submonoid (with respect to the multiplication in $K(G)$)  of $K(G)$ generated by $f_1, \ldots,f_n$, and finally for $f_1, \ldots,f_n \in G$ by $gp(f_1, \ldots,f_n)$  we denote  the subgroup  of $G$ generated by $f_1, \ldots,f_n$.

\begin{theorem} \label{th:subring} 
Let $G$ be a non-abelian limit group with  a finite generating set $X$, $K$ an infinite field. Then the following hold for any $n \in \N$:
\begin{enumerate}
\item there is a formula $Group_n(h,f_1, \ldots,f_n, X)$ with parameters $X$ (uniformly in $X$) such that for any elements $f_1, \ldots,f_n, h \in G$   
$$K(G) \models Group_n(h,f_1, \ldots, f_n,X) \Longleftrightarrow h \in gp(f_1, \ldots, f_n);$$

\item there is a formula $Mon_n(h,f_1, \ldots,f_n, X)$ with parameters $X$ (uniformly in $X$) such that for any elements $f_1, \ldots,f_n, h \in G$   
$$K(G) \models Mon_n(h,f_1, \ldots, f_n,X) \Longleftrightarrow h \in mon(f_1, \ldots, f_n);$$

 If $K$ is interpretable in $\N$, then  elements $f_1, \ldots,f_n$ in the above statements can be also  taken in $K(G)$. Moreover,
\item  there is a formula $Ring_n(h,f_1, \ldots,f_n, X)$ with parameters $X$ (uniformly in $X$) such that for any elements $f_1, \ldots,f_n, h \in K(G)$  
$$K(G) \models Ring_n(h,f_1, \ldots, f_n,X) \Longleftrightarrow h \in ring(f_1, \ldots, f_n);$$
\item  there is a formula $Ideal_n(y,y_1, \ldots,f_n, X)$ with parameters $X$ (uniformly in $X$) such that for any elements $f_1, \ldots,f_n, h \in K(G)$   
$$K(G) \models Ideal_n(h,f_1, \ldots, f_n,X) \Longleftrightarrow h \in id(f_1, \ldots, f_n).$$

\end{enumerate}
\end{theorem} 
\begin{proof}
We will show how to prove the first statement. We have a series of interpretations (by lemma \ref{le:list-superstructure} and Theorem \ref{th:bi1})
$$G\rightarrow _{int}S(\Z,\N) \rightarrow _{bi-int} \N \rightarrow _{bi-int}S(\Z,\N) \rightarrow _{int} K(G)$$The subgroup  $gp(f_1, \ldots, f_n)$ is recursively enumerable. Therefore the set of all tuples $W=(f_1,\ldots ,f_n,h)$ such that  $h\in gp(f_1,...,f_n) $ is also recursively enumerable. Therefore the  image of it $W^{\diamond}=\{(t_1,\ldots t_n,t)\}$ in $S(\Z,\N)$ and  the image $W^{\diamond\diamond}=\{(s_1,\ldots s_n,s)\}\subseteq \N^{n+1}$ is also recursively enumerable. This image is a recursively enumerable set of tuples of integers . In $\N$ every recursively enumerable relation  is diophantine. Namely, by  \cite{Mat}
we have  

$$ (s_1,\ldots s_n,s)\in W^{\diamond\diamond} \iff \exists (x_1,...x_m) \ P(s_1,...s_n,s,x_1,...,x_m)=0,$$ where $P(s_1,...s_n,s,x_1,...,x_m)$ is a polynomial with integer coefficients. Therefore the set  $W^{\diamond\diamond}$ is defined by a formula in $\N$. Since we have bi-interpretation of $\N$ with $S(\Z,\N)$,  the corresponding set $W^{\diamond}$ in $S(\Z,\N)$ is also defined by a formula that we denote $\psi(t_1,\ldots t_n,t).$ Let $\bar t_1,\ldots ,\bar t_n,\bar h$ be the images of $t_1,\ldots ,t_n,t$  when we interpret $S(\Z,\N)$ in $K(G).$

By Lemma \ref{Hodges},  there is a formula $\psi ^{*}(\bar t_1,\ldots ,\bar t_n,\bar h)$ in the language of $K(G)$ such that 
 $S(\Z,\N) \models \psi(t_1,\ldots t_n,t)$ if and only if $K(G)
\models \psi^*(\bar t_1,\ldots \bar t_n,\bar t).$
By Theorem \ref{th:bi1}, there is a formula $\phi (g,\bar t)$   defining the isomorphism from $G$ to its image in $K(G)$  obtained in a series of interpretations 
above. Let $$\Phi (f_1,\ldots ,f_n,h)=\exists \bar t_1,\ldots ,\bar t_n,\bar t (\psi ^{*}(\bar t_1,\ldots ,\bar t_n,\bar t) \wedge _{i=1}^n \phi (f_i,\bar t_i)\wedge \phi (h,\bar t))$$  

Then $S(\Z,\N)\models
\psi (t_1,...,t_n,t)$ if and only if  $K(G)\models \Phi (f_1,...,f_n,h)$ if and only if $h \in gp(f_1,...,f_n)$, so we have
our result.

The other three statements of the theorem can be proved using similar argument because the structures $S(K,\N)$ and $K(G)$ are bi-interpretable by Theorem \ref{th:bi1} and $K$ is interpretable in $\N$ (so $S(K,\N)$ is interpretable in $\N$).

\end{proof}

\subsection{Rings elementarily equivalent to group algebras of free and limit groups}

 \begin{theorem} \label{th:eeq} Let $F$ be a finitely generated free group and $K$ an infinite field. Let  $H$ be a group, $L$  be a field, such that there is an element in $H$ with finitely generated centralizer  in $L(H)$. Then $K(F) \equiv L(H)$ if and only if 
 \begin{itemize}
 \item  $H$ is isomorphic to $F$, 
 \item $HF(K) \equiv HF(L)$.
 \end{itemize}
\end{theorem}
\begin{proof} Let $B=L(H)$.  Then by Theorem \ref{th:hyp} $H\equiv F$ and $HF(K) \equiv HF(L)$.

All  centralizers of elements from $H$ in $B=L(H)$ are definably isomorphic to each other as rings. Furthermore, the ring isomorphic to such a centralizer,  say $C$ is 0-interpretable in $B$ by the same formulas that the ring of one-variable Laurent polynomials $K[t, t^{-1}]$ is interpretable in $K(F)$, it follows from the properties of 0-interpretations 
 that $C \equiv K[t, t^{-1}]$. Since at least one  centralizer of an element in $H$ in $B$ is Noetherian by the assumption, the ring $C$ is Noetherian.  The ring $C$ is isomorphic to $K_1[t, t^{-1}]$  for some field $K_1$ as a  noetherian ring elementarily equivalent to the ring of Laurent polynomials over a field. Since $K_1$ is the set of all invertible (and $0$) elements $x$ in $C$ such that $x+1$ is also invertible, it follows that $L = K_1$.  We showed that every  centralizer of an element in $H$ in $B$ is isomorphic to $L[t, t^{-1}]$. 

Lemma \ref{natural} gives a uniform interpretation   of arithmetic $\N$ in $B$.

All statements of Proposition  \ref{th:S(F,N)-non-comm} also hold in $B$, and the corresponding interpretations are given precisely by the same formulas as in $K(F)$. This gives  interpretations of $S(L, \N)$ in $B$ which satisfy all the statements of Proposition \ref{th:S(F,N)-non-comm}.

A direct  analog of Lemma \ref{le:M-X-c} and  Lemma \ref{le:bi-int-iso} holds in $B$. Indeed, Lemma \ref{le:M-X-c} tells us that there is a definable isomorphism between the subgroup $H$ generated by some set $X$ in $B$ and the free subgroup $\MM_{X,P}$  canonically interpreted in $S(\Z,\N)$.  Since the analog of Proposition \ref{th:S(F,N)-non-comm} holds in $B$ the same formulas as in Lemma \ref{le:M-X-c}  interpret $S(\Z,\N)$ in $B$. Therefore, the subgroup $H$ in $B$ is definably isomorphic to  the free group $\MM_{X,P}$. In particular,  $H$ is a free group generated by $X$. 

This proves the theorem.  

\end{proof}
Actually a more general result holds.
 \begin{theorem} \label{th:eeq2} Let $F$ be a finitely generated free group and $K$ an infinite field. Let  $B$ be a ring that has a non-central invertible element $x$ with noetherian  centralizer.  If $K(F) \equiv B$ then 
 \begin{itemize}
 \item $B$ is isomorphic to a group algebra $L(H)$,
 \item  $H$ is isomorphic to $F$,
 \item $HF(K) \equiv HF(L)$.
 \end{itemize}
\end{theorem}
The following Theorem is a corollary.
 \begin{theorem} \label{th:12} Let $K$ be an infinite field and $F_n$ a free non-abelian group of rank $n$. Then
   $K(F_n) \not  \equiv K(F_m)$ for $n \neq m$.
\end{theorem}
Similar results hold for group algebras of limit groups.
 \begin{theorem} \label{th:eeq2} Let $G$ be a non-abelian  limit group and $K$ an infinite field. Let $L$  be a field and  $H$ be a group such that each  non-trivial element in $H$  has finitely generated centralizer  in $L(H)$. Then $K(G) \equiv L(H)$ if and only if 
 \begin{itemize}
 \item  $H$ is isomorphic to $G$, 
 \item $HF(K) \equiv HF(G)$.
 \end{itemize}
\end{theorem}
\begin{proof} Limit groups are left orderable. Let $B=L(H)$.  Then by Theorem \ref{th:bi1} $H\equiv F$ and $HF(K) \equiv HF(L)$.

All  centralizers of elements from $H$ in $B=L(H)$ are abelian and belong to a finite family of  isomorphism classes. Since they all are noetherian, for each centralizer $C$ of a non-central element in $B$  we have $C \equiv K[T, T^{-1}]$, where $T$ is a finite set of variables.  The ring $C$ is isomorphic to $K_1[T, T^{-1}]$  for some field $K_1$ as a  noetherian ring elementarily equivalent to the ring of Laurent polynomials over a field. Since $K_1$ is the set of all invertible (and $0$) elements $x$ in $C$ such that $x+1$ is also invertible, it follows that $L = K_1$.  We showed that every  centralizer of an element in $H$ in $B$ is isomorphic to $L[T, T^{-1}]$, for some finite set $T$.  By Lemma \ref{le:limit} the ring of Laurent polynomials with one variable is definable in each centralizer, and all the copies of the ring of Laurent polynomials with one variable  inside different centralizers are definably isomorphic. Lemma \ref{natural} gives a uniform interpretation   of arithmetic $\N$ in $B$.

We have interpretations of $S(L, \N)$ in $B$ and  by Theorem \ref{th:bi1}
  there is a definable isomorphism between the subgroup $H$ generated by some set $X$ in $B$ and the limit group $G$   canonically interpreted in $S(L,\N)$.  Since the analog of Proposition \ref{th:S(F,N)-non-comm} holds in $B$, similar  formulas as in Lemma \ref{le:M-X-c}  interpret $S(L,\N)$ in $B$. Therefore, the subgroup $H$ in $B$ is definably isomorphic to  the limit group $G$.
This proves the theorem.  

\end{proof}

 \begin{theorem} \label{th:limit} Let $G$ be a non-abelian  limit group and $K$ an infinite field. Let  $B$ be a ring such that all non-central invertible elements in $B$ have  noetherian centralizers. If $K(F) \equiv B$ then 
 \begin{itemize}
 \item $B$ is isomorphic to a group algebra $L(H)$,
 \item  $H$ is isomorphic to $F$,
 \item $HF(K) \equiv HF(L)$.
 \end{itemize}
\end{theorem}

\section{Quantifier elimination}\label{elimination}

In this section we study an important model theoretic property of group rings of limit groups: quantifier elimination.

Let $L$ be a first-order language. Recall that a  formula $\phi$ in $L$ is in a prenex normal form if $\phi = Q_1 y_1Q_2y_2 \ldots Q_sy_s\phi_0(x_1, \ldots,x_m)$ where $Q_i$ are quantifiers ($\forall$ or $\exists$), and $\phi_0$ is a quantifier-free formula in $L$. It is known that every formula in $L$ is equivalent to a formula in the prenex normal form.
A formula $\phi = Q_1 y_1Q_2y_2 \ldots Q_sy_s\phi_0(x_1, \ldots,x_m)$ in the prenex normal form is called $\Sigma_n$ formula if the sequence of quantifiers  $ Q_1 Q_2 \ldots Q_s$ begins with the  existential quantifier $\exists$  and alternates  $n-1$  times between series of existential and universal quantifiers. Similarly,  a formula $\phi$ above is $\Pi_n$ formula if the sequence of quantifiers  $ Q_1 Q_2 \ldots Q_s$ begins with the  universal  quantifier $\forall$  and alternates $n-1$  times between series of existential and universal quantifiers. 

For a structure  $\mathcal S$  of the language $L$ denote by $\Sigma_n(\mathcal S)$ the  set of all subsets of ${\mathcal S}^m$, $m \in \N$, definable in $\mathcal S$ by  $\Sigma_n$ formulas  $\phi(x_1, \ldots,x_m)$, $m \in \N$. Replacing in the definition above  $\Sigma_n$ by $\Pi_n$  one gets the set $\Pi_n(\mathcal S)$.  Let  $\Sigma_0(\mathcal S) = \Pi_0(\mathcal S)$ be the set of all subsets definable in $\mathcal S$ by quantifier-free formulas.  Clearly,
$$
\Sigma_0(\mathcal S) \subseteq  \Sigma_0(\mathcal S) \subseteq \ldots \subseteq \Sigma_n(\mathcal S) \subseteq \ldots
$$
$$
\Pi_0(\mathcal S) \subseteq  \Pi_0(\mathcal S) \subseteq \ldots \subseteq \Pi_n(\mathcal S) \subseteq \ldots
$$
 The sets $\Sigma_n(\mathcal S)$ and $\Pi_n(\mathcal S)$ form the so-called {\em arithmetical hierarchy} oven $\mathcal S$ denoted by ${\mathcal H}(\mathcal S)$.  It is easy to see that if $\Sigma_n (\mathcal S) = \Sigma_{n+1}(\mathcal S)$ (or $\Pi_n (\mathcal S) = \Pi_{n+1}(\mathcal S)$) for some $n \in \N$ then $\Sigma_m (\mathcal S) = \Sigma_{m+1}(\mathcal S)$ and $\Pi_m (\mathcal S) = \Pi_{m+1}(\mathcal S)$ for every natural $m \geq n$.  We say that the hierarchy ${\mathcal H}(\mathcal S)$ {\em collapses} if $\Sigma_n (\mathcal S) = \Sigma_{n+1}(\mathcal S)$ for some $n \in \N$, otherwise it is called {\em proper}. 
 
 A unitary ring $R$ can be viewed in the ring language $+,\cdot,0,1$ and also in the ring language extended by new constants $c_r$ for each element $r\in R$. In the former case the arithmetical hierarchy is denoted by ${\mathcal H}(R)$, and in the latter one the arithmetical hierarchy is denoted by ${\mathcal H}^R(R)$.

\begin{theorem}
\label{th:elimination}
Let $G$ be a non-abelian limit group and $K$ an infinite field which is interpretable in $\N$. Then for any $n \in \N$ there is a formula $\phi_n$ of the first-order language of rings such that the formula $\phi_n$ is not equivalent in $K(G)$ to any boolean combination of formulas from $\Pi_n$ or $\Sigma_n$ (with constants from $K(G)$).
\end{theorem}
\begin{proof} Since $K$ is interpretable in $\N$ it follows that $S(K,\N)$ is interpretable in $\N$ and, therefore, $S(K,\N)$ is bi-interpretable with $\N$. Since $K(G)$ is bi-interpretable with $S(K,\N)$, $K(G)$ is bi-interpretable with $\N$.

Suppose, to the contrary, that for some $n \in \N$ any formula  $\phi (\bar x)$ in the ring language $L$ is equivalent in $K(G)$ to some boolean combination  $\phi'(\bar x)$ of formulas from $\Pi_n$ or $\Sigma_n$ with constants from $K(G)$. Take  an arbitrary first-order formula $\psi(\bar z)$ of the language of $\N$. Since $K(G)$ is bi-interpretable in $\N$ the  formula $\psi(\bar z)$ can be rewritten into  a formula $\phi (\bar x)$ of the ring language such that for any values $\bar a$ of $\bar z$, $\N \models \psi(\bar a) \Longleftrightarrow K(G) \models \phi (\bar b)$, where $\bar a\rightarrow\bar b$ when $\N$ is interpreted in $K(G)$. By our assumption there is a formula $\phi'(\bar x)$, which is a boolean combination   of formulas from $\Pi_n$ or $\Sigma_n$ perhaps with constants from $K(G)$ such that $\phi(\bar x)$ is equivalent to $\phi'(\bar x)$ in $K(G)$. Since $K(G)$ is bi-interpretable in $\N$  there is a number $m$ which depends only on the bi-interpretation such that $\phi'(\bar x)$ can be rewritten into a formula $\psi'(\bar z)$, which is a boolean combination of formulas from $\Pi_{n+m}$ or $\Sigma_{n+m}$  in the language of $\N$ such that $\N \models \psi'(\bar a) \Longleftrightarrow K(G) \models \phi'(\bar b)$. It follows that  $\psi (\bar z)$ is equivalent to $\psi '(\bar z)$ in $\N$, i.e., every formula $\psi$ of the language of $\N$ is equivalent in $\N$ to some formula $\psi '$ which is a boolean combination of formulas from $\Pi_{n+m}$ in the language of $\N$. However, this is false since the arithmetical hierarchy in $\N$ is proper. It follows that our assumption is false, so the theorem holds.
\end{proof}
\begin{cor}
Let $K$ be an infinite field which is interpretable in $\N$ and $G$ a non-abelian limit group. Then the hierarchies ${\mathcal H}(K(G))$ and ${\mathcal H}^{K(G)}(K(G))$ are proper.
\end{cor}

 We would like to thank the referee for carefully reading our manuscript and for giving  constructive comments which substantially helped improving the quality of the paper.

\end{document}